\documentclass[12pt,a4paper]{amsart} 
\usepackage[utf8]{inputenc}
\usepackage[english]{babel}

\usepackage{amsmath}
\usepackage{amsfonts}
\usepackage{amssymb}
\usepackage{color}

\usepackage{amsthm}

\newtheorem{theorem}{Theorem} [section]

\newtheorem{corollary}[theorem]{Corollary} 
\newtheorem{lemma}[theorem]{Lemma}

\newtheorem{open}{Open Question}

\usepackage{graphics}
\usepackage{epsfig}

\usepackage{url}
\usepackage{hyperref}

\usepackage[a4paper,top=2.5cm,bottom=2.8cm,
left=3.2cm,right=3.2cm]{geometry}
\markleft{\hfill \textsc{Natural Orderings of Triangle Centers} \hfill}
\newcommand{\degrees}{^\circ}

\long\def\void#1{}

\begin{document}
\void{
International Journal of  Computer Discovered Mathematics (IJCDM) \\
ISSN 2367-7775 \copyright IJCDM \\
Volume 11, 2026 pp. xx--yy  \\
web: \url{http://www.journal-1.eu/} \\
Received xx Mar. 2026. Published on-line xx Mmm 2026\\ 

\copyright The Author(s) This article is published 
with open access.\footnote{This article is distributed under the terms of the Creative Commons Attribution License which permits any use, distribution, and reproduction in any medium, provided the original author(s) and the source are credited.} \\
}

\bigskip
\bigskip

\begin{center}
	{\Large \textbf{Natural Orderings of Triangle Centers} }\\
	\medskip
	\bigskip
        \bigskip

	\textsc{Stanley Rabinowitz} \\

	545 Elm St Unit 1,  Milford, New Hampshire 03055, USA \\
	e-mail: \href{mailto:stan.rabinowitz@comcast.net}{stan.rabinowitz@comcast.net}
	\\web: \url{http://www.StanleyRabinowitz.com} \\

\bigskip
\bigskip

\end{center}
\bigskip

\textbf{Abstract.}
Triangle centers are usually studied individually or through special geometric relationships, but little attention has been given to global structure among them. In this paper we introduce several natural ways to order triangle centers, including the isosceles order, vertex order, side order, and trace order. These partial orders compare centers by their relative positions in families of triangles, such as acute triangles with a fixed shortest side. Using barycentric coordinates and symbolic computation, we determine ordering relations among many of the first 100 triangle centers listed in Kimberling's \emph{Encyclopedia of Triangle Centers}. The results reveal surprising structural patterns and suggest new ways to organize and study triangle centers. Many new inequalities are also revealed. For example, in an acute triangle $ABC$, with shortest side $BC$, the Gergonne point is always closer to side $BC$ than the nine-point center.



\newenvironment{code}[2]
{
\medskip
\hspace{#1}%
\begin{minipage}{#2}
\color{blue}
}
{
\color{black}
\smallskip
\end{minipage}%
}

\newcommand{\redtext}[1]
{\textcolor{red}{#1}
}

\bigskip
\bigskip
%
\section{Introduction}
\label{section:introduction}

Triangle centers form one of the richest collections of special points associated with a triangle. Thousands of such centers have been cataloged in Kimberling's \emph{Encyclopedia of Triangle Centers} \cite{ETC}, and many remarkable geometric relationships between them are known. Most studies, however, focus on individual centers or small families of centers.

In this paper we investigate a different question: can triangle centers be ordered in a natural way?

When a triangle varies in shape, the positions of its centers move in complicated ways, often appearing chaotic. Nevertheless, when we restrict attention to certain families of triangles, clear ordering relationships emerge. We introduce several ways to compare triangle centers based on their geometric positions, including comparisons by distance from a vertex, distance from a side, and the position of cevian traces.

Let $X_n$ denote the $n$-th triangle center cataloged in \cite{ETC}.

Figure~\ref{fig:10000ThinCenters} shows a random triangle. The black dots represent the triangle centers
$X_1$, $X_2$, $X_3,\ldots, X_{10000}$, except that only triangle centers of $\triangle ABC$ are
shown if they lie inside the triangle.

\begin{figure}[h!t]
\centering
\includegraphics[width=0.5\linewidth]{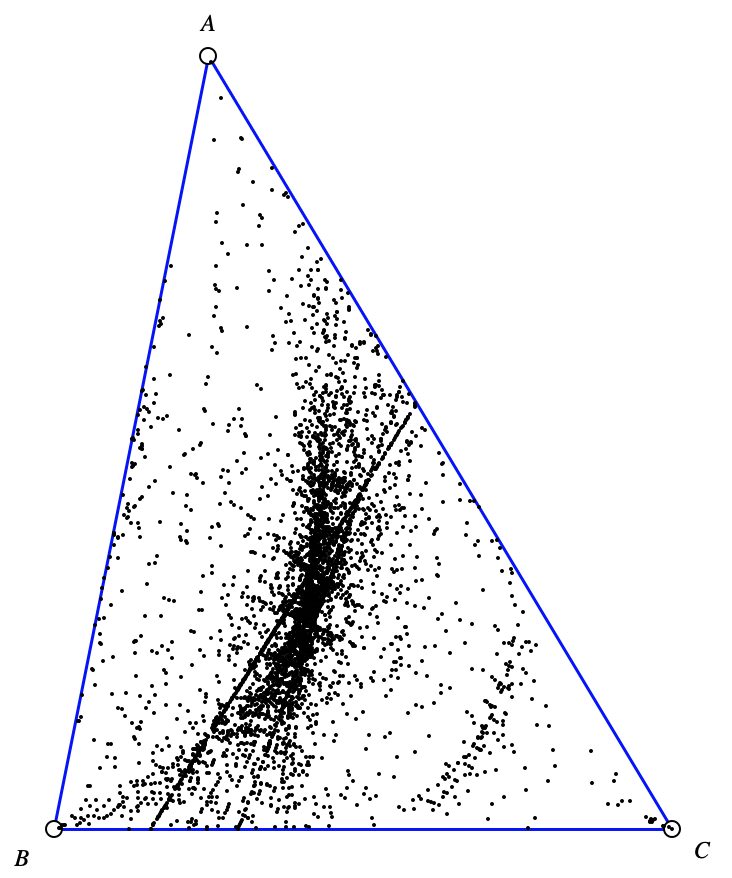}
\caption{Triangle Centers inside $\triangle ABC$}
\label{fig:10000ThinCenters}
\end{figure}

Some patterns can be noticed. For example, a number of dots lie along
straight lines. These correspond to known central lines~\cite{CentralLines} associated with a triangle,
such as the Euler Line, the Nagel Line, and the Brocard Axis.
Apart from this pattern, the dots appear to be randomly scattered throughout the interior of the triangle.
As we vary the shape of the triangle, the dots move around seemingly chaotically as they change their order
in relationship to each other. Sometimes the points move outside the triangle.

In this paper, we try to bring some order to this chaos.
We explore various ways for ordering these centers.

To study these orderings we use barycentric coordinates together with symbolic computation in Mathematica. Section 2 develops several preliminary results concerning barycentric coordinates and distances to the sides of a triangle. Sections 3--6 introduce four different orderings on triangle centers and investigate their properties. In each case we determine ordering relationships among many of the first 100 triangle centers listed in \cite{ETC}.

The reader who just wants to learn about ordering triangle centers
can skip Section~\ref{section:prelim} and proceed to Section~\ref{section:iso}.

%

\section{Preliminary Remarks}
\label{section:prelim}


We assume that the reader is familiar with barycentric coordinates.
We use the standard convention that in $\triangle ABC$, $a=BC$, $b=CA$, and $c=AB$.

The complete Mathematica proofs of all the theorems in this paper can be
found in the Mathematica notebook included in the supplementary material
associated with this paper.


If a point $P$ has barycentric coordinates $(p:q:r)$ and $p+q+r=1$,
then we say that the coordinates are \emph{normalized}. 
(Yiu \cite[\S 3.1.2]{Yiu} calls these absolute barycentric coordinates.)

Since barycentric coordinates use signed areas, if the normalized barycentric
coordinates for a point $P$ are $(p:q:r)$, the value of $p$ is
positive if $P$ lies on the same side of $BC$ as $A$, and negative if it lies on the opposite side.
The value of $q$ is
positive if $P$ lies on the same side of $CA$ as $B$, and negative if it lies on the opposite side.
The value of $r$ is
positive if $P$ lies on the same side of $AB$ as $C$, and negative if it lies on the opposite side.
This gives us the following lemma.

\begin{lemma}
\label{lemma:7areas}
The sidelines of a triangle divide the plane into seven regions.
If the normalized barycentric coordinates for a point $P$ are $(p:q:r)$,
then the signs of $p$, $q$, and $r$ are as shown in Figure~\ref{fig:7areas}.
\end{lemma}

\begin{figure}[h!t]
\centering
\includegraphics[width=0.6\linewidth]{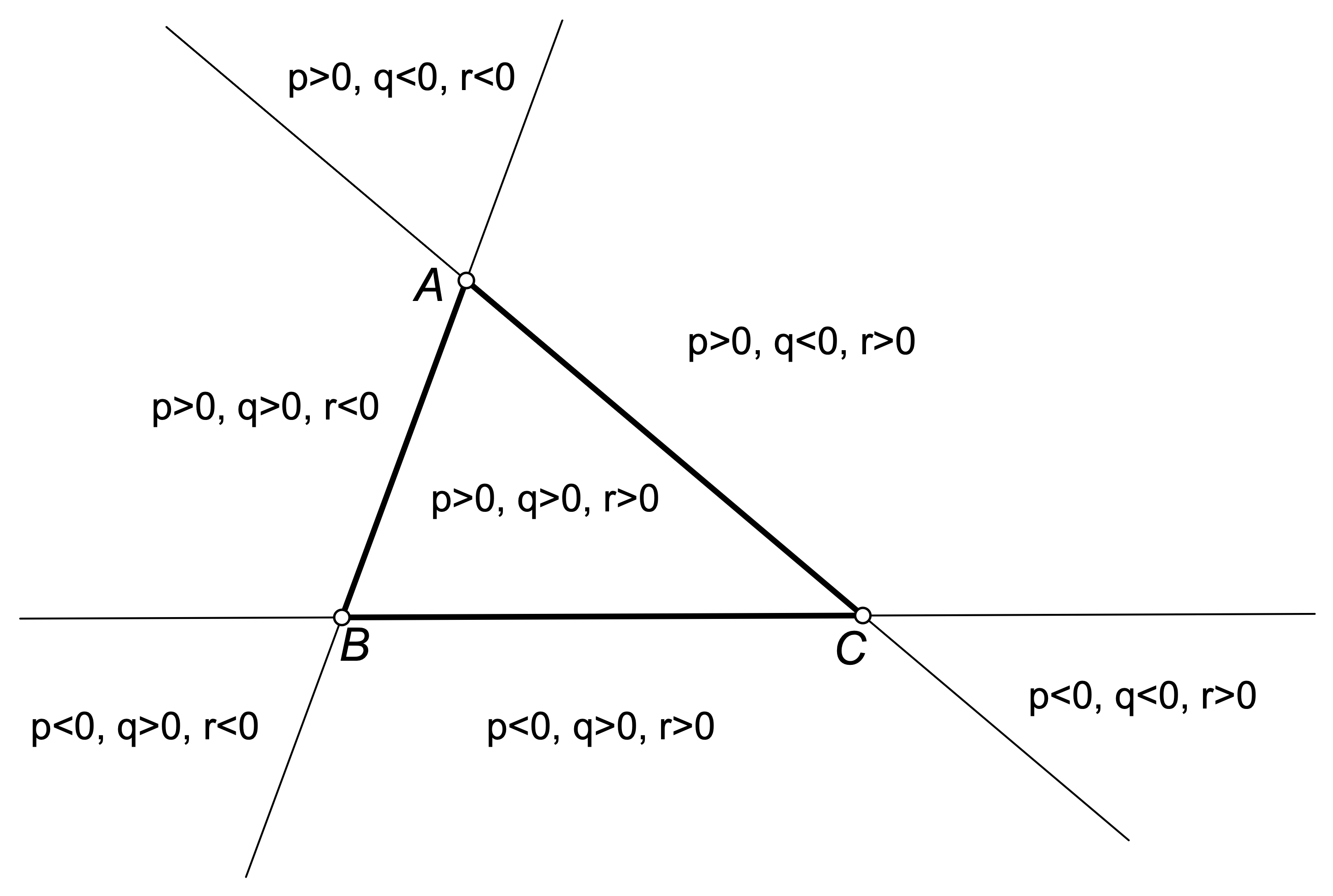}
\caption{Seven regions formed by the sidelines of a triangle}
\label{fig:7areas}
\end{figure}

In $\triangle ABC$, a point is said to be \emph{inside} angle $A$ if it
lies in the open convex region bounded by rays $\overrightarrow{AB}$ and $\overrightarrow{AC}$.
From Lemma~\ref{lemma:7areas}, we get the following result.

\begin{theorem}
Let $P$ have barycentric coordinates $(u:v:w)$ with respect to $\triangle ABC$.
Then $P$ is inside angle $A$ if and only if
$$v(u+v+w)>0\qquad\textrm{and}\qquad w(u+v+w)>0.$$
\end{theorem}

\begin{proof}
Since barycentric coordinates are homogeneous, the coordinates $(u:v:w)$ and $(uk:vk:wk)$
represent the same point for any $k\neq 0$.
The normalized barycentric coordinates for $P$ are therefore
$$\left(\frac{u}{u+v+w}:\frac{v}{u+v+w}:\frac{w}{u+v+w}\right).$$
Note that $u+v+w\neq 0$ since $P$ is not on the line at infinity.
From Figure~\ref{fig:7areas}, we see that for normalized coefficients, $P=(p:q:r)$ is inside angle $A$
if and only if $q>0$ and $r>0$.
In terms of $u$, $v$, and $w$, this means that
$v/(u+v+w)>0$ and $w/(u+v+w)>0$.
Multiplying each inequality by the positive quantity $(u+v+w)^2$ preserves
the sense of the inequality and gives us the desired result.
\end{proof}

\begin{theorem}
\label{thm:insideIso}
Let $ABC$ be an isosceles triangle with base $BC$.
Then $X_n$ sometimes lies outside angle~$A$ when\\
$n\in \{4, 5, 11, 13, 14, 16, 17, 18, 19, 22, 23, 24, 25, 26, 27, 28, 29, 30,
33, 34, 36, 44,$\\
$46, 47, 48, 49, 50, 51, 52, 53, 54, 59, 62, 64, 66, 67,
68, 70, 73, 74, 77, 79, 80, 84, 87,$\\
$88, 90, 91, 92, 93, 94, 95, 96,
97, 98, 99, 100\}$.\\
The center $X_{30}$ lies on the line at infinity.\\
In all other cases, $X_n$ always lies inside angle $A$.
\end{theorem}

\begin{proof}
We will use Mathematica. Let \texttt{Xn=\{p,q,r\}} where $(p:q:r)$ denotes the barycentric
coordinates for center $X_n$ as found from \cite{ETC}.

The following Mathematica code creates a function called \texttt{sometimeOutsideQ}
such that \texttt{sometimeOutsideQ[Xn]} returns \texttt{True} if and only if
center $X_n$ sometimes lies outside angle~$A$.

\begin{code}{0.4in}{5in}
\begin{verbatim}
triang = a>0 && b>0 && c>0 && a+b>c && b+c>a && c+a>b;
isosceles = (b == c);
constraint = triang && isosceles;
sometimeOutsideQ[{u_, v_, w_}] := Module[{inst},
   inst = FindInstance[(v(u+v+w) <= 0 || w(u+v+w) <= 0)
             && constraint, {a,b,c}];
   If[inst =!= {}, Return[True]];
   Return[False];
   ];
\end{verbatim}
\end{code}

We then use this function on each $X_n$ to determine those $n$ for which
$X_n$ can lie outside angle~$A$.
\end{proof}

\begin{theorem}
\label{thm:inside}
Let $ABC$ be an isosceles triangle with base $BC$ and $b>a$.
Then $X_n$ lies inside angle~$A$ for $1\leq n\leq 100$ except for $X_n$
when\\
$n\in \{18,  26,  30, 59, 68, 70, 87, 90, 91, 93, 96, 99, 100\}$.
\end{theorem}

\begin{proof}
The proof is the same as the proof of Theorem~\ref{thm:insideIso}, except that the constraint
is replaced by the following Mathematica code.

\begin{code}{0.4in}{5in}
\begin{verbatim}
constraint = triang && isosceles && (b>a);
\end{verbatim}
\end{code}

Again, each $X_n$ is tested with routine \texttt{sometimeOutsideQ[Xn]}.
\end{proof}

\begin{theorem}
Let $ABC$ be an isosceles triangle with base $BC$.
Then $X_n$ coincides with vertex $A$ when $n\in \{59, 99, 100\}$.
\end{theorem}

\begin{proof}
We use the following Mathematica code.

\begin{code}{0.4in}{5in}
\begin{verbatim}
ptA= {1,0,0};
Simplify[Cross[Xn, ptA] == {0,0,0}, constraint]
\end{verbatim}
\end{code}

The \texttt{Simplify} command then returns \texttt{True} if and only if $X_n$ coincides with vertex~$A$.
\end{proof}

\begin{corollary}
Let $ABC$ be an isosceles triangle with base $BC$ and $b>a$.
Then for $1\leq n\leq 100$, $X_n$ always lies outside angle~$A$ only when
$X_n$ coincides with vertex~$A$.
\end{corollary}

In a 1--2--2 triangle, $X_n$ lies outside angle $A$ for $n\in\{18,26,68,93\}$.\\
In a 11--16--16 triangle, $X_n$ lies outside angle $A$ or $n\in\{90,91,96\}$.\\
In a 7--16--16 triangle, $X_{87}$ lies outside angle $A$.\\
In a 1--2--2 triangle, $X_{87}$ coincides with vertex $A$.\\

In an isosceles triangle, $X_{70}$ can lie outside vertex angle $A$ but only if $a>b$.\\
In a 9--8--8 triangle, $X_{70}$ lies outside angle $A$.\\

In $\triangle ABC$, a point is said to be \emph{above} side $BC$ if it
lies on the same side of $BC$ as vertex $A$.
From Lemma~\ref{lemma:7areas}, we get the following result.

\begin{theorem}
\label{thm:above}
Let $P$ have barycentric coordinates $(u:v:w)$ with respect to $\triangle ABC$.
Then $P$ is above side $BC$ if and only if
$$u(u+v+w)>0.$$
\end{theorem}

\begin{proof}
Since barycentric coordinates are homogeneous, the coordinates $(u:v:w)$ and $(uk:vk:wk)$
represent the same point for any $k\neq 0$.
The normalized barycentric coordinates for $P$ are therefore
$$\left(\frac{u}{u+v+w}:\frac{v}{u+v+w}:\frac{w}{u+v+w}\right).$$
Note that $u+v+w\neq 0$ since $P$ is not on the line at infinity.
From Figure~\ref{fig:7areas}, we see that for normalized coefficients, $P=(p:q:r)$ is above side $BC$
if and only if $p>0$.
In terms of $u$, $v$, and $w$, this means that
$u/(u+v+w)>0$.
Multiplying this inequality by the positive quantity $(u+v+w)^2$ preserves
the sense of the inequality and gives us the desired result.
\end{proof}

\begin{theorem}
\label{thm:aboveOrder}
Let $ABC$ be an acute triangle with smallest side $BC$.
Then $X_n$ always lies above $BC$ when\\
$n\in \{1, 2, 3, 4, 5, 6, 7, 8, 9, 10, 12, 13, 15, 17, 19, 20, 21, 22, 25,
26, 27, 28, 29, 31, 32,$\\
$33, 34, 35, 37, 38, 39, 40, 41, 42, 45, 48,
51, 53, 54, 55, 56, 57, 58, 59, 60, 61, 63, 65, 68,$\\
$69, 71, 72, 73, 75, 76, 77, 78, 79, 81, 82, 83, 85, 86, 89, 92, 95, 97, 99, 100\}$.\\
In addition, $X_n$ always lies on or above $BC$ when $n=11$.\\
Center $X_n$ lies on the line at infinity when $n=30$.\\
Center $X_n$ always lies below $BC$ when $n\in \{16, 23, 36, 44, 50\}$.\\
Center $X_n$ can lie above or below $BC$ when
$n\in \{14,18,24,43,46,47,49,52,62,$\\
$64,66,67,70,74,80,84,87,88,90,91,93,94,96,98\}$.
\end{theorem}

\begin{theorem}[Distance From Point to Line]
\label{thm:distToLine}
The distance from the point\goodbreak $P=(p:q:r)$ to the line $ux+vy+wz=0$ is
$$\sqrt{\frac{\left(2a^2b^2+2b^2c^2+2c^2a^2-a^4-b^4-c^4\right) (p u+q v+r w)^2}{4 (p+q+r)^2 \left(a^2 (u-v) (u-w)+(v-w)
   \left(b^2 (v-u)+c^2 (u-w)\right)\right)}}.$$
\end{theorem}

\begin{proof}
The equation of the line through $P$ perpendicular to the line $L$ with equation  $ux+vy+wz=0$ is
given by formula (7) in \cite{Grozdev}. The intersection of these two lines can be found using
the intersection formula (5) in  \cite{Grozdev}. This gives us the coordinates for $F$, the foot
of the perpendicular from $P$ to $L$. Then, using the formula (9) from \cite{Grozdev} for the
distance between two points, we get the distance from $P$ to $F$ which is the distance from point $P$
to line $L$.
\end{proof}

\begin{corollary}
\label{cor:distToLine}
The distance from point $P=(p:q:r)$ to side $BC$ of $\triangle ABC$ is
$$\sqrt{\frac{p^2 \left(2a^2b^2+2b^2c^2+2c^2a^2-a^4-b^4-c^4\right)}{4 a^2 (p+q+r)^2}}.$$
\end{corollary}

\begin{proof}
This follows from the fact that the equation for the line $BC$ is $x=0$,
so we can put $u=1$, $v=0$, and $w=0$ in Theorem~\ref{thm:distToLine}.
\end{proof}

\begin{corollary}
\label{cor:distToLineK}
The distance from point $P=(p:q:r)$ to side $BC$ of $\triangle ABC$ is
$$\frac{2K}{a}\sqrt{\frac{p^2}{(p+q+r)^2}}=\frac{2K}{a}\left|\frac{p}{p+q+r}\right|$$
where $K$ is the area of $\triangle ABC$.
\end{corollary}

\begin{proof}
This follows because the formula
$$K=\frac{1}{4}\sqrt{2a^2b^2+2b^2c^2+2c^2a^2-a^4-b^4-c^4}$$
follow algebraically from Heron's formula, $K=\sqrt{s(s-a)(s-b)(s-c)}$ when
$s=(a+b+c)/2$.
\end{proof}

\begin{corollary}
The distance from vertex $A$ to side $BC$ of $\triangle ABC$ is
$$\frac{2K}{a}$$
where $K$ is the area of $\triangle ABC$.
\end{corollary}

\begin{proof}
The coordinates for point $A$ are $(1:0:0)$, so let $p=1$, $q=0$, and $r=0$ in Corollary~\ref{cor:distToLineK}.
This also follows from the fact that the area of a triangle is half the base times the height.
\end{proof}

\textbf{Definition.}
Let $d(P)$ denote the \emph{signed distance} from a point $P$ in the plane of $\triangle ABC$ to side $BC$.
The sign is positive if $P$ is above $BC$ and negative if $P$ is below $BC$.

\begin{theorem}
\label{thm:signedDistance}
The signed distance from point $P=(p:q:r)$ to side $BC$ of $\triangle ABC$ is
$$d(P)=\frac{2K}{a}\left(\frac{p}{p+q+r}\right).$$
\end{theorem}

\begin{proof}
The (unsigned) distance is
$$\frac{2K}{a}\left|\frac{p}{p+q+r}\right|$$
by Corollary~\ref{cor:distToLineK}.
If $P$ is above $BC$, then by Theorem~\ref{thm:above}, $p(p+q+r)>0$.
Dividing both sides by the positive quantity $(p+q+r)^2$ shows that $p/(p+q+r)>0$.
Thus, the formula for $d(p)$ is correct when $P$ is above $BC$. Similarly,
if $P$ is below $BC$, Theorem~\ref{thm:above} implies that $p(p+q+r)<0$
or equivalently $p/(p+q+r)<0$.
Thus, the formula for $d(p)$ is correct when $P$ is below $BC$.
If $P$ lies on $BC$, then $p=0$ and the formula is also correct.
\end{proof}

%

\section{The Isosceles Order}
\label{section:iso}

From the definition of a triangle center, it is easy to see that in an isosceles triangle,
all triangle centers lie on the median to the base.
At first glance, it appears that the order of these triangle centers along the median is
essentially random and varies as the shape of the triangle changes.
However, if we only consider ``tall'' isosceles triangles, that is, ones
in which $b>a$, we find a little more order.

\goodbreak

\textbf{Definition.}
An isosceles triangle $ABC$ with base $BC$ is said to be \emph{tall} if $b>a$.

\textbf{Definition.}
We define the \emph{isosceles order} on triangle centers, $\prec$, by

\centerline{$P\prec Q$ if $P$ is closer to $A$ than $Q$}
\centerline{in all tall isosceles triangles $ABC$ with base $BC$.}

The symbols ``$P\prec Q$'' can be read as ``$P$ precedes $Q$'' or less precisely as ``$P$ is less than $Q$''.

Under this ordering, we find some order among the triangle centers as shown by the following theorem
which involves some of the triangle centers from $X_1$ to $X_{30}$.

\begin{theorem}
\label{thm:isosceles}
Using the isosceles order $\prec$, we have
$$X_{20}\prec  X_{22}\prec  X_8 \prec  X_3\prec  X_9\prec  X_{10}\prec  X_{21}\prec  X_2\prec  X_5\prec  X_{12}\prec  X_{17}\prec  X_1\prec  X_{13}$$
$$\prec  X_7\prec  X_6\prec X_4\prec  X_{27}\prec  X_{19}\prec  X_{28}\prec  X_{25}\prec  X_{11}\prec  X_{14}\prec  X_{16}\prec  X_{23}.$$
\end{theorem}

\begin{proof}
Let $d_n$ denote the square of the distance from $X_n$ to $A$.
Let us start with the first claim, $X_{20}\prec  X_{22}$.
Both $X_{20}$ and $X_{22}$ lie inside $\angle A$ (Theorem~\ref{thm:inside}).
Thus, all we need to show is that $d_{20}<d_{22}$.
Using the distance formula \cite[\S2]{Grozdev}, we find that
$$d_{20}=-\frac{a^6-3 a^2 \left(b^2-c^2\right)^2+2 \left(b^2-c^2\right)^2
   \left(b^2+c^2\right)}{a^4-2 a^2 \left(b^2+c^2\right)+\left(b^2-c^2\right)^2}.$$
When $c=b$, this simplifies to
$$d_{20}=\frac{a^4}{(2 b-a) (2 b+a)}.$$
Similarly,
$$d_{22}=\frac{a^2 b^2 c^2 \left(-a^8+2 a^6 \left(b^2+c^2\right)-2 a^2 \left(b^2-c^2\right)^2
   \left(b^2+c^2\right)+\left(b^4-c^4\right)^2\right)}{\left(a^6-a^4
   \left(b^2+c^2\right)-a^2 \left(b^4+c^4\right)+\left(b^2-c^2\right)^2
   \left(b^2+c^2\right)\right)^2}$$
and when $c=b$, this simplifies to
$$d_{22}=\frac{a^4 b^4 (2 b-a) (2 b+a)}{\left(a^4-2 a^2 b^2-2 b^4\right)^2}.$$
The inequality $d_{20}<d_{22}$ becomes
$$\frac{a^4}{(2 b-a) (2 b+a)}<\frac{a^4 b^4 (2 b-a) (2 b+a)}{\left(a^4-2 a^2 b^2-2 b^4\right)^2}.$$
We could prove this inequality using algebra, standard inequality techniques,
and some ingenuity, but instead, it is easier
to use a symbolic algebra system to prove the inequality.

We will use the Mathematica command\\
\centerline{\texttt{Simplify[expr,constraint]}}
which simplifies an expression, equation, or inequality, subject to the stated constraint.
If we let\\
\centerline{\texttt{isoConstraint = 0<a<b \&\& a+b>c \&\& b+c>a \&\& c+a>b \&\& b==c}}
and issue the Mathematica command\\
\centerline{\texttt{Simplify[d[20]<d[22],isoConstraint]},}
Mathematica responds with \texttt{True}, proving that the inequality is always true.

Using the \texttt{Simplify} command in Mathematica, all the other claims can be proved in the same manner.
\end{proof}

Figure~\ref{fig:isosceles30} shows a tall isosceles triangle and the centers
referenced in Theorem~\ref{thm:isosceles}.
As the shape of the triangle varies, the distances between these centers change relative to each other,
however the order of the centers remains fixed.

\begin{figure}[h!t]
\centering
\includegraphics[width=0.4\linewidth]{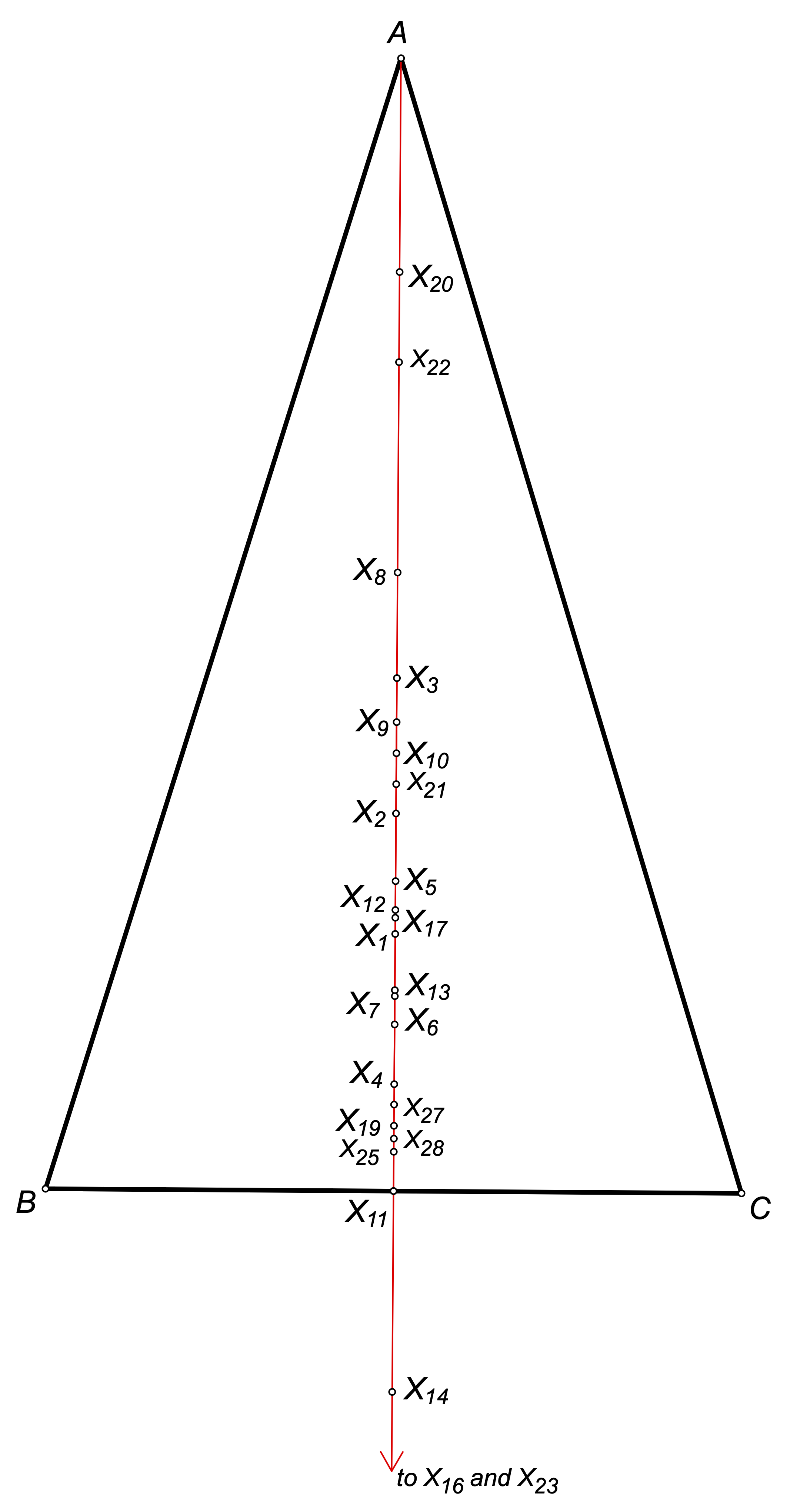}
\caption{Centers in a tall isosceles triangle}
\label{fig:isosceles30}
\end{figure}

You will note that triangle centers $X_{15}$, $X_{18}$, $X_{24}$, $X_{26}$, $X_{29}$, and $X_{30}$ were not included in Theorem~\ref{thm:isosceles}.
This is because these centers are not always in the same order with respect to the other triangle centers
or they lie outside angle $A$.
Let us look a little closer at the location of these centers.

\goodbreak
\textbf{Center $X_{15}$.}

Figure~\ref{fig:X15} shows that $X_{15}$ is sometimes closer to $A$ than $X_5$
and sometimes further away.

\begin{figure}[h!t]
\centering
\includegraphics[width=0.5\linewidth]{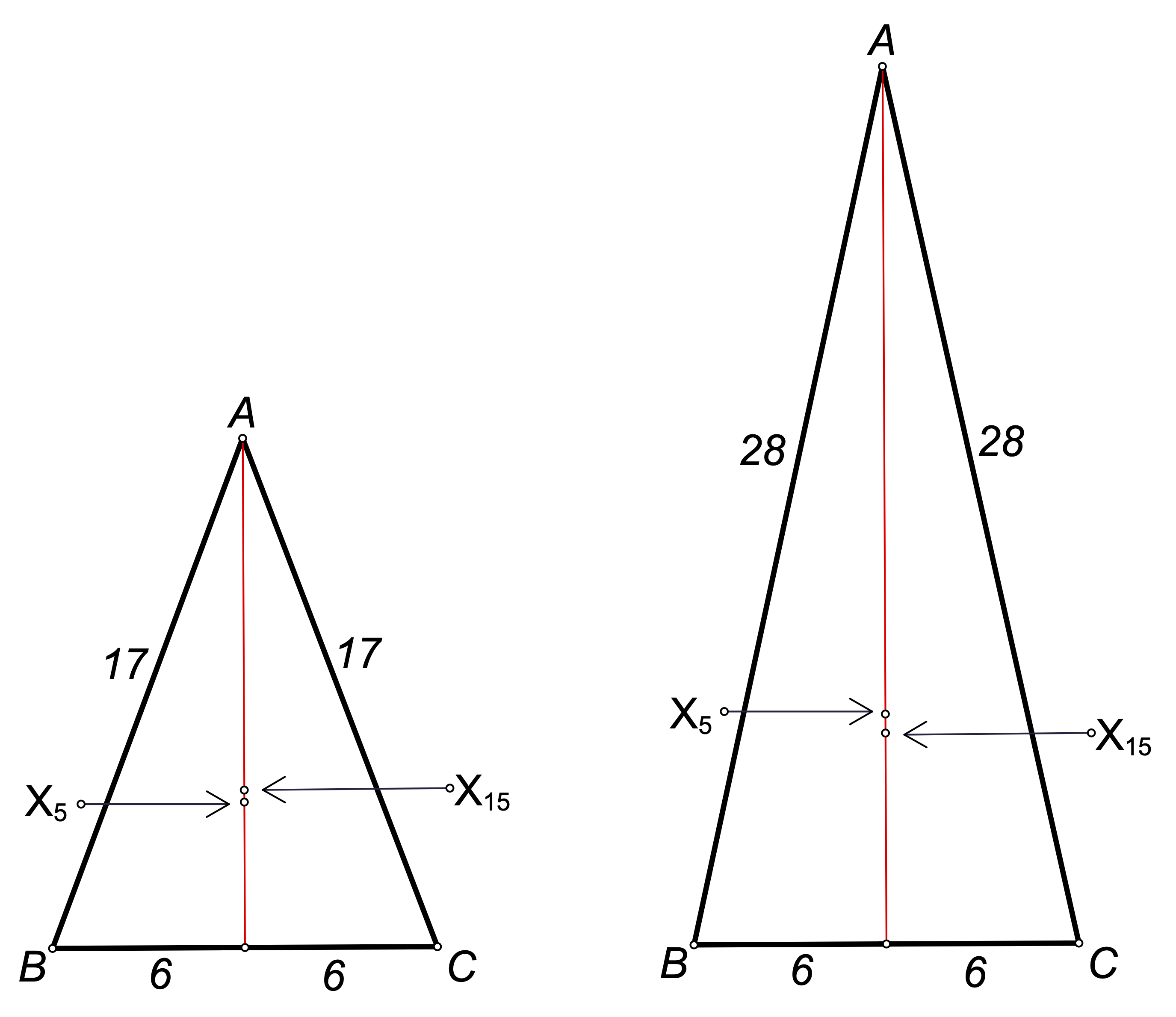}
\caption{Centers $X_{15}$ and $X_{5}$}
\label{fig:X15}
\end{figure}

Center $X_{15}$ is closer to $A$ than $X_{12}$ in a 1--4--4 triangle and further from $A$ in a 1--8--8 triangle.
We can, however, bound the location of $X_{15}$ as the following result shows.

\begin{theorem}
\label{thm:X15}
Let $ABC$ be a tall isosceles triangle with base $BC$.
Then
$$X_2\prec  X_{15} \prec  X_{17}.$$
\end{theorem}

\begin{figure}[h!t]
\centering
\includegraphics[width=0.25\linewidth]{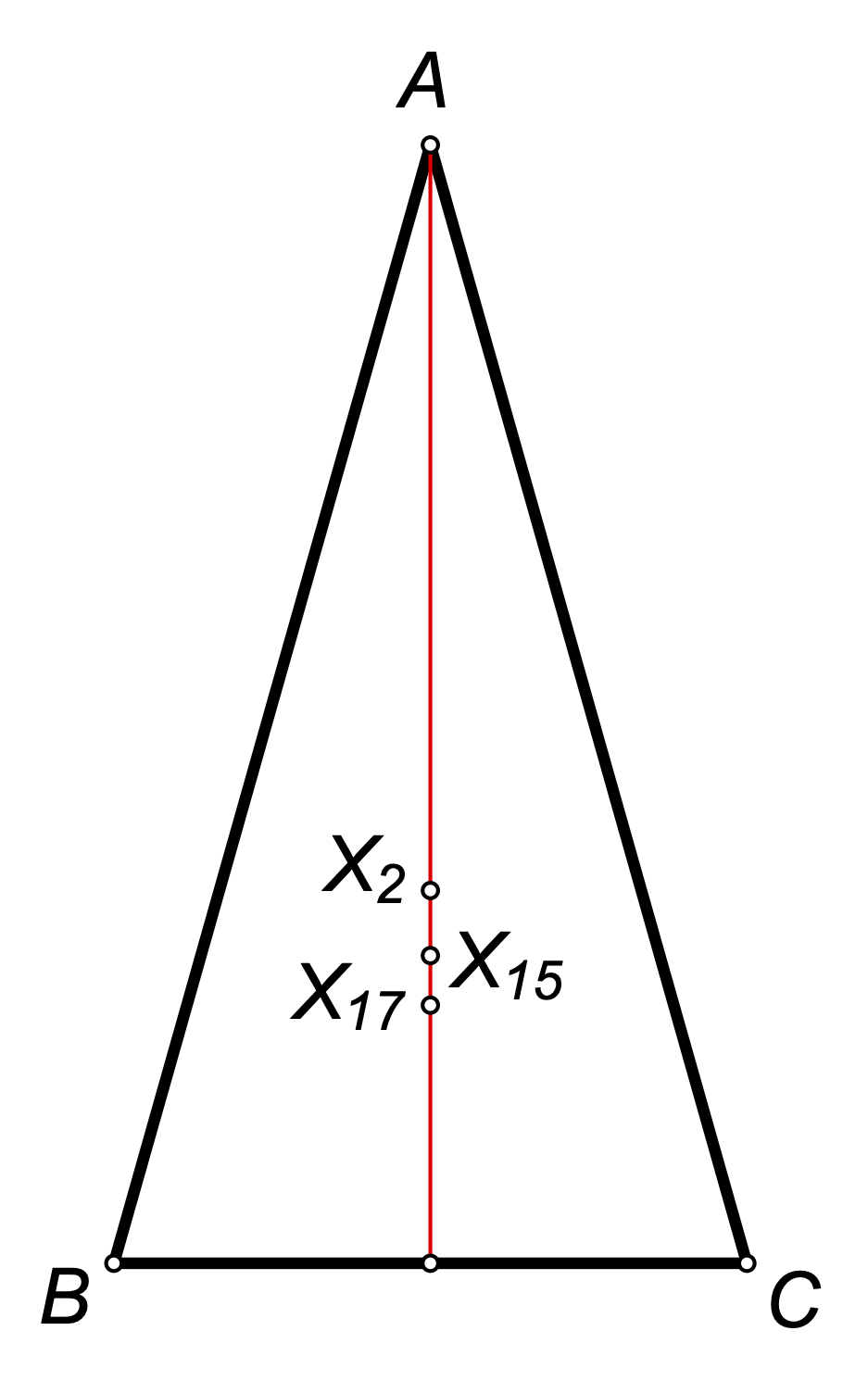}
\caption{Centers $X_2$, $X_{15}$, and $X_{17}$}
\label{fig:X2X15X17}
\end{figure}

\begin{proof}
The result can be proven using the Mathematica \texttt{Simplify} command
in the same way that it was used in the proof of Theorem~\ref{thm:isosceles}.
\end{proof}

\begin{theorem}
\label{thm:5/15}
Triangle centers $X_5$ and $X_{15}$ coincide in a 1--$k$--$k$ triangle,
where $k=\sqrt{2+\sqrt3}$.
\end{theorem}

\begin{proof}
Two barycentric representations $p=(p_1:p_2:p_3)$ and $q=(q_1:q_2:q_3)$ represent the same point if
their coordinates are proportional. In Mathematica, this condition is \texttt{Cross[p,q]==\{0,0,0\}}.

Using this fact and the barycentric coordinates for $X_5$ and $X_{15}$ from \cite{ETC}, we find
that \texttt{Cross[X5,X15]==\{0,0,0\}} when $a=1$ and $k=\sqrt{2+\sqrt3}$.
\end{proof}

\begin{theorem}
Triangle centers $X_{12}$ and $X_{15}$ coincide in a 1--$k$--$k$ triangle,
where $k\approx 7.25054$ is the positive real root of $2 x^6-13 x^5-11 x^4+6 x^2+x-1=0$.
\end{theorem}

\begin{proof}
This example is obtained using the Mathematica command\\
\centerline{\texttt{FindInstance[Cross[X12,X15]==\{0,0,0\} \&\& b==c>a==1, \{a,b,c\}]}}\\
where \texttt{X12} and \texttt{X15} are the barycentric coordinates for centers $X_{12}$ and $X_{15}$,
respectively.
\end{proof}

\textbf{Center $X_{18}$.}

Figure~\ref{fig:X18} shows that $X_{18}$ is sometimes closer to $A$ than $X_{20}$
and sometimes further away.

\begin{figure}[h!t]
\centering
\includegraphics[width=0.5\linewidth]{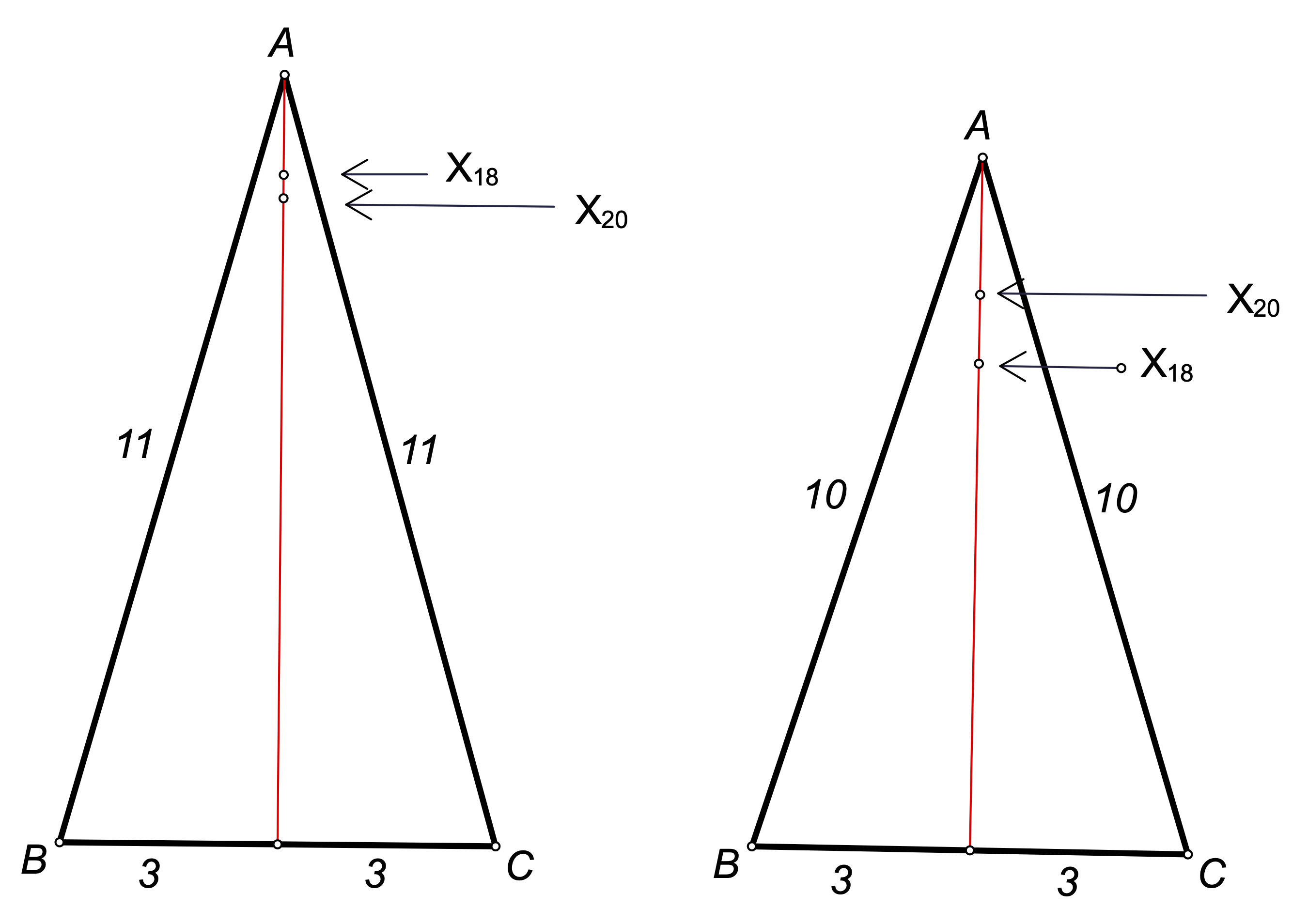}
\caption{Centers $X_{18}$ and $X_{20}$}
\label{fig:X18}
\end{figure}

Figure~\ref{fig:X18X26} shows that $X_{18}$ and $X_{26}$
can lie outside angle $A$.
Sometimes $X_{18}$ is closer to $A$ than $X_{26}$ and sometimes it is further away.

\begin{figure}[h!t]
\centering
\includegraphics[width=0.3\linewidth]{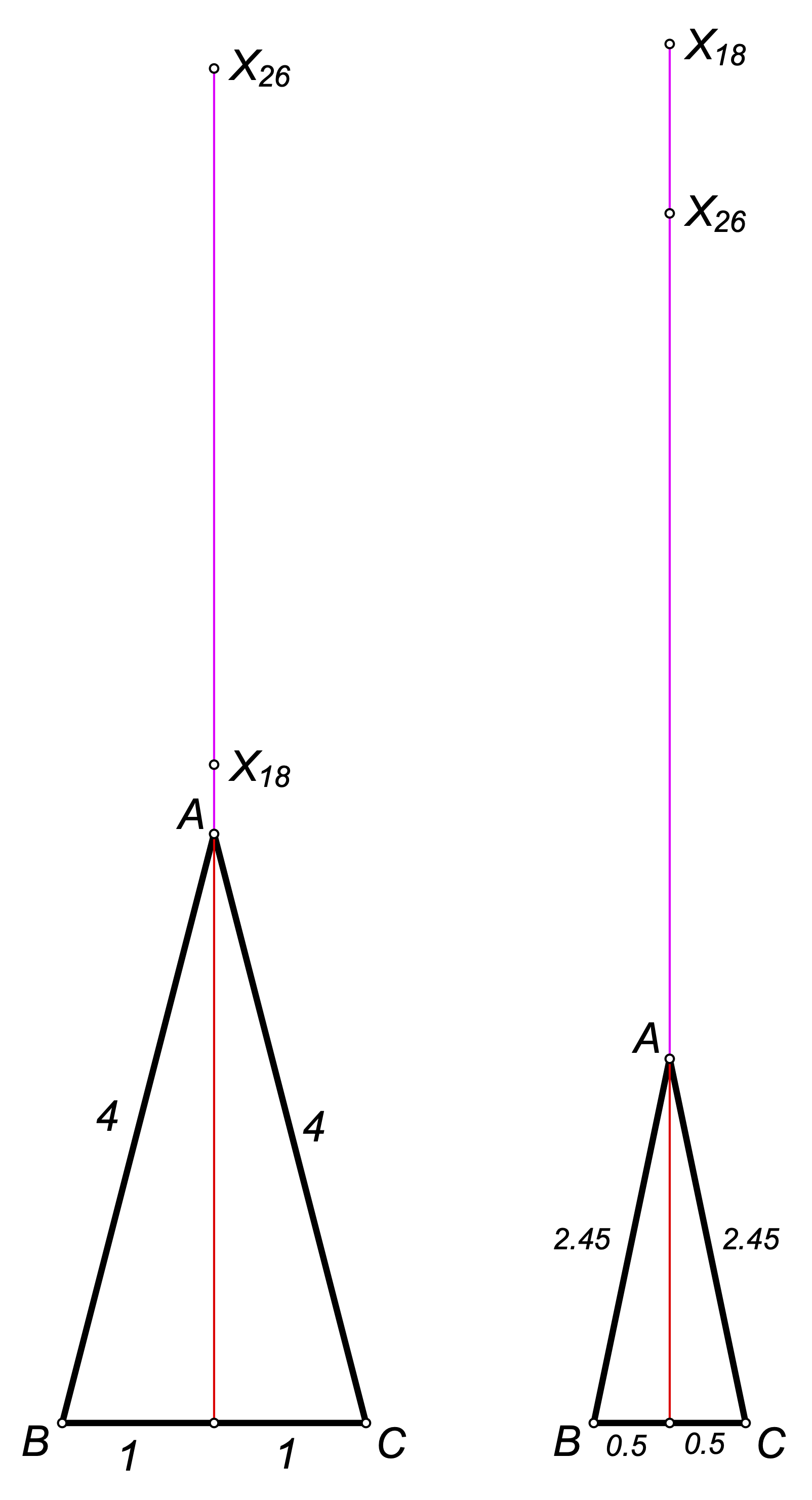}
\caption{Centers $X_{18}$ and $X_{26}$}
\label{fig:X18X26}
\end{figure}

\goodbreak
\textbf{Center $X_{24}$.}

Figure~\ref{fig:X11X24} shows that $X_{24}$ is sometimes closer to $A$ than $X_{11}$
and sometimes further away.

\begin{figure}[h!t]
\centering
\includegraphics[width=0.4\linewidth]{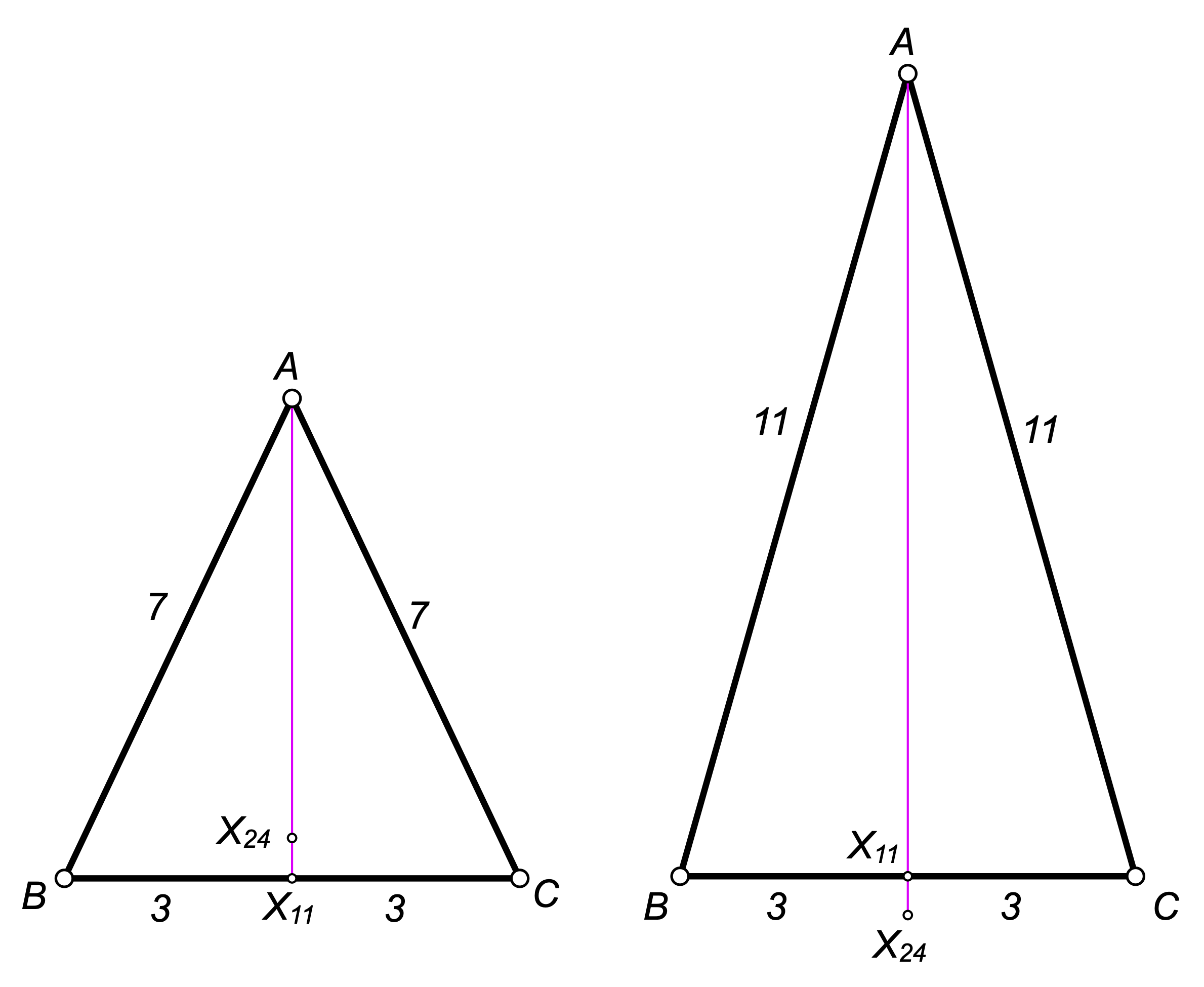}
\caption{Centers $X_{11}$ and $X_{24}$}
\label{fig:X11X24}
\end{figure}

\begin{theorem}
Triangle centers $X_{11}$ and $X_{24}$ coincide in a $k$--1--1 triangle,
where $k=\sqrt{2-\sqrt2}$.
\end{theorem}

\begin{proof}
The proof is the same as the proof of Theorem~\ref{thm:5/15}.
\end{proof}

We can, however, bound the location of $X_{24}$ as the following result shows.

\begin{theorem}
\label{thm:X24}
Let $ABC$ be a tall isosceles triangle with $b=c$.
Then
$$X_{25}\prec  X_{24} \prec  X_{14}.$$
\end{theorem}

\begin{proof}
The proof is the same as the proof of Theorem~\ref{thm:X15}.
\end{proof}

\textbf{Center $X_{26}$.}

Figure~\ref{fig:X18X26} shows that center $X_{26}$ can sometimes lie outside angle $A$.

If $X_{26}$ lies inside angle $A$, then it lies between $A$ and $X_{20}$ (Figure~\ref{fig:X20X26}).

\begin{figure}[h!t]
\centering
\includegraphics[width=0.35\linewidth]{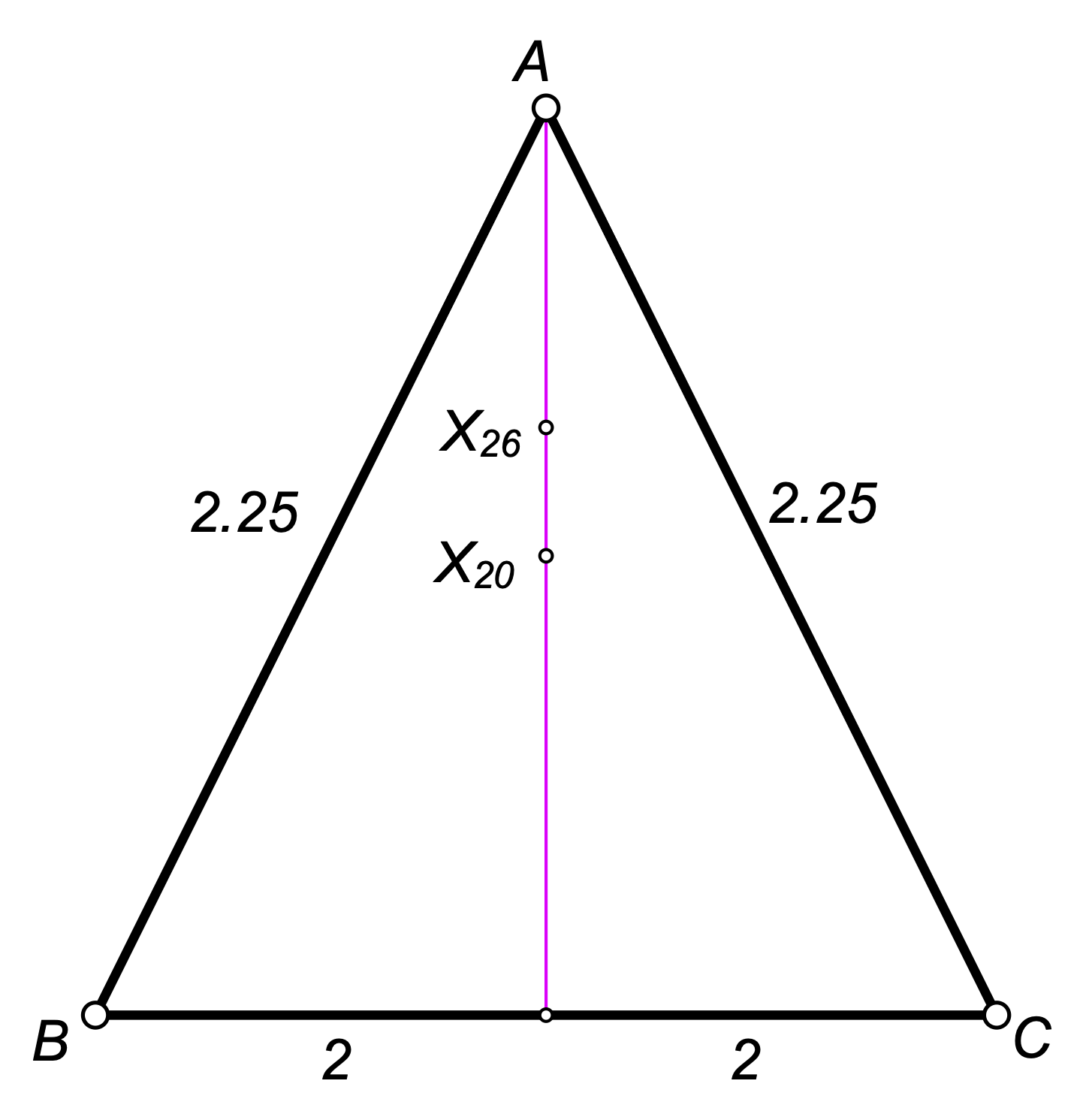}
\caption{Centers $X_{20}$ and $X_{26}$}
\label{fig:X20X26}
\end{figure}

\begin{theorem}
Triangle center $X_{26}$ coincides with vertex $A$ in a $k$--1--1 triangle,
where $k=\sqrt{2-\sqrt2}$.
\end{theorem}

\begin{proof}
The proof is the same as the proof of Theorem~\ref{thm:5/15}.
\end{proof}

\textbf{Center $X_{29}$.}

Center $X_{29}$ is closer to $A$ than $X_{6}$ in a 1--2--2 triangle and further from $A$ in a 7--8--8 triangle.
We can, however, bound the location of $X_{29}$ as the following result shows.

\begin{theorem}
\label{thm:X29}
Let $ABC$ be a tall isosceles triangle with $b=c$ .
Then
$$X_{7}\prec  X_{29} \prec  X_{4}.$$
\end{theorem}

\begin{proof}
The proof is the same as the proof of Theorem~\ref{thm:X15}.
\end{proof}

\begin{theorem}
Triangle centers $X_{29}$ and $X_{6}$ coincide in a 1--$k$--$k$ triangle,
where $k=(1+\sqrt{17})/4$.
\end{theorem}

\begin{proof}
The proof is the same as the proof of Theorem~\ref{thm:5/15}.
\end{proof}

\textbf{Center $X_{30}$.}

\begin{theorem}
Center $X_{30}$ lies on the line at infinity.
\end{theorem}

\begin{proof}
From \cite{ETC}, we have
$$X_{30}=\Bigl(2 a^4-a^2 \left(b^2+c^2\right)-\left(b^2-c^2\right)^2:$$
$$\qquad 2 b^4-b^2 \left(c^2+a^2\right)-\left(c^2-a^2\right)^2:$$
$$\qquad 2 c^4-c^2 \left(a^2+b^2\right)-\left(a^2-b^2\right)^2\Bigr).$$
Algebraically, the sum of the coefficients is 0 which means that the point lies on the line at infinity.
\end{proof}

\bigskip
\textbf{Center $X_{11}$.}

\begin{theorem}
\label{thm:X11}
The Feuerbach point, $X_{11}$, of an isosceles triangle coincides
with the midpoint of the base.
\end{theorem}

\begin{proof}
Without loss of generality, assume $b=c$.
From \cite{ETC}, we find that the barycentric coordinates for $X_{11}$ are
$$\left((b - c)^2 ( b + c-a): (a - c)^2 (c+a - b ): (a - b)^2 (a + b - c)\right).$$
Letting $c=b$ gives
$$X_{11}=\left(0: a (a - b)^2: a (a - b)^2\right)=(0:1:1),$$
so $X_{11}$ is the midpoint of $BC$.
\end{proof}


Using the information obtained above, we can draw a graph showing the ordering relationship
between various centers in a tall isosceles triangle.
The graph is shown in Figure~\ref{fig:isoscelesGraph30},
where an arrow from m to n means that $X_m\prec X_n$ in the isosceles order.

\begin{figure}[h!t]
\centering
\includegraphics[width=0.2\linewidth]{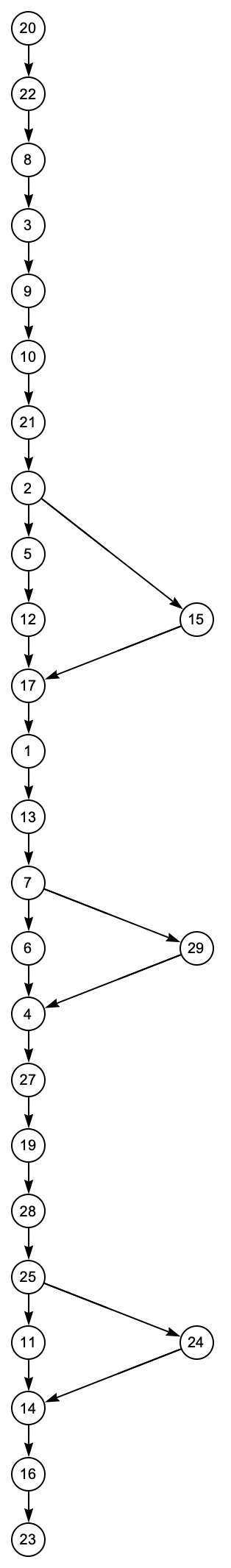}
\caption{An arrow from m to n means $X_m\prec X_n$ in the isosceles order.}
\label{fig:isoscelesGraph30}
\end{figure}


%
%

If we look at the first 100 triangle centers in \cite{ETC}, instead of the first 30, we get the following result.

\begin{theorem}
\label{thm:isosceles100}
Let $ABC$ be a tall isosceles triangle with $b=c$.
Then, using the isosceles order $\prec$, we have
$$X_{20}\prec X_{22}\prec X_{40}\prec X_{72}\prec X_{63}\prec X_{8}\prec X_{3}\prec X_{9}\prec X_{95}\prec X_{77}\prec X_{21}\prec X_{2}$$
$$\prec X_{45}\prec X_{38}\prec X_{55}\prec X_{37}\prec X_{12}
\prec X_{17}\prec X_{1}\prec X_{61}\prec X_{60}\prec X_{81}\prec X_{7}\prec X_{82}$$
$$\prec X_{89}\prec X_{6}\prec X_{65}\prec X_{33}\prec X_{51}\prec X_{57}\prec X_{4}\prec X_{27}\prec X_{19}
\prec X_{28}\prec X_{25}\prec X_{34}$$
$$\prec X_{64}\prec X_{11}\prec X_{98}\prec X_{74}\prec X_{67}\prec X_{88}\prec X_{14}\prec X_{80}\prec X_{36}
\prec X_{16}\prec X_{44}\prec X_{23}.$$
\end{theorem}

Not all centers from $X_1$ through $X_{100}$ appear in this list because some centers
vary their order relative to other centers. The complete order for all centers
that lie in angle $A$ is shown in Figure~\ref{fig:isoscelesGraph100}.

\begin{figure}[h!t]
\centering
\includegraphics[width=0.56\linewidth]{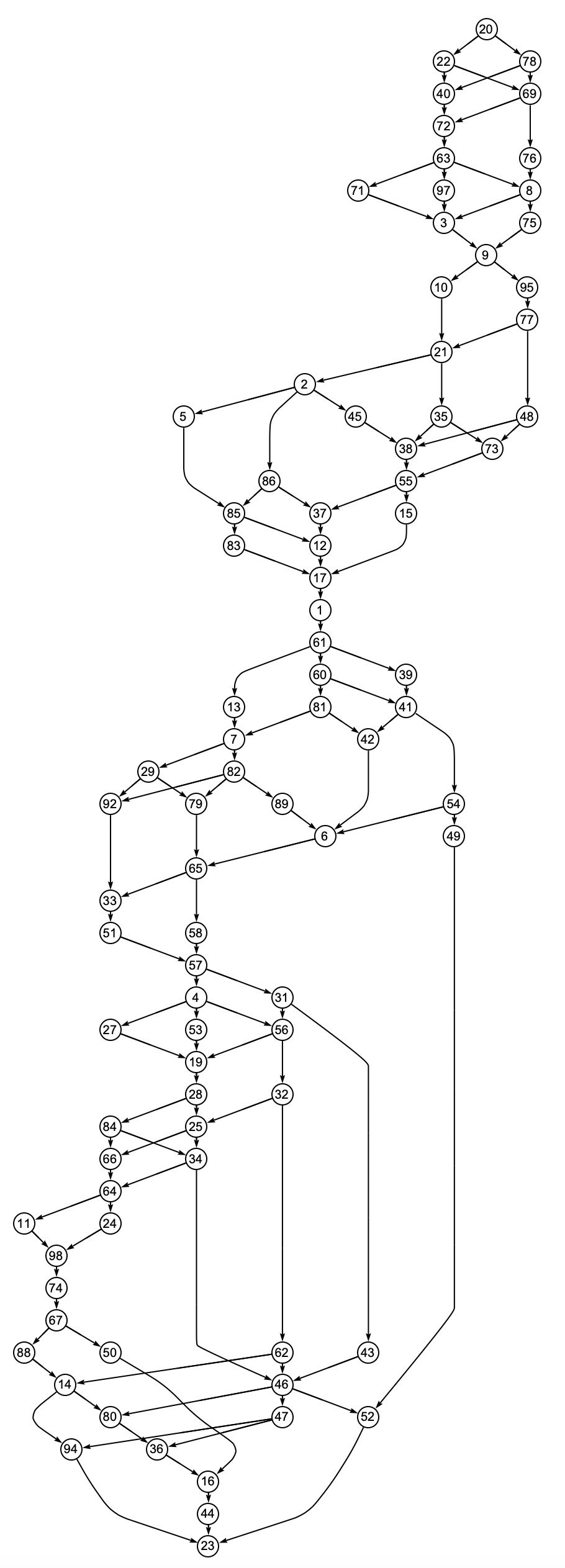}
\caption{An arrow from m to n means $X_m\prec X_n$.}
\label{fig:isoscelesGraph100}
\end{figure}

\begin{open}
Is there a simple reason that center $X_1$ is a cutpoint \cite{MathWorld-cutpoint}
for the graph shown in Figure~\ref{fig:isoscelesGraph100}?
\end{open}

\begin{open}
As more centers are added to the graph in Figure~\ref{fig:isoscelesGraph100},
does $X_1$ remain a cutpoint?
\end{open}

\newpage

%

\section{The Vertex Order}

\textbf{Definition.}
We define the \emph{vertex order} on triangle centers, $\prec$, by

\centerline{$P\prec Q$ if $P$ is closer to $A$ than $Q$}
\centerline{for all acute triangles $ABC$ with shortest side $BC$.}

\begin{theorem}
\label{thm:vertex}
Using the vertex order $\prec$, we have (Figure~\ref{fig:vertexOrder})
$$X_3\prec  X_9\prec  X_{10}\prec  X_2\prec  X_1\prec X_6\prec X_4\prec  X_{19}\prec  X_{16}.$$
\end{theorem}

\begin{figure}[h!t]
\centering
\includegraphics[width=0.6\linewidth]{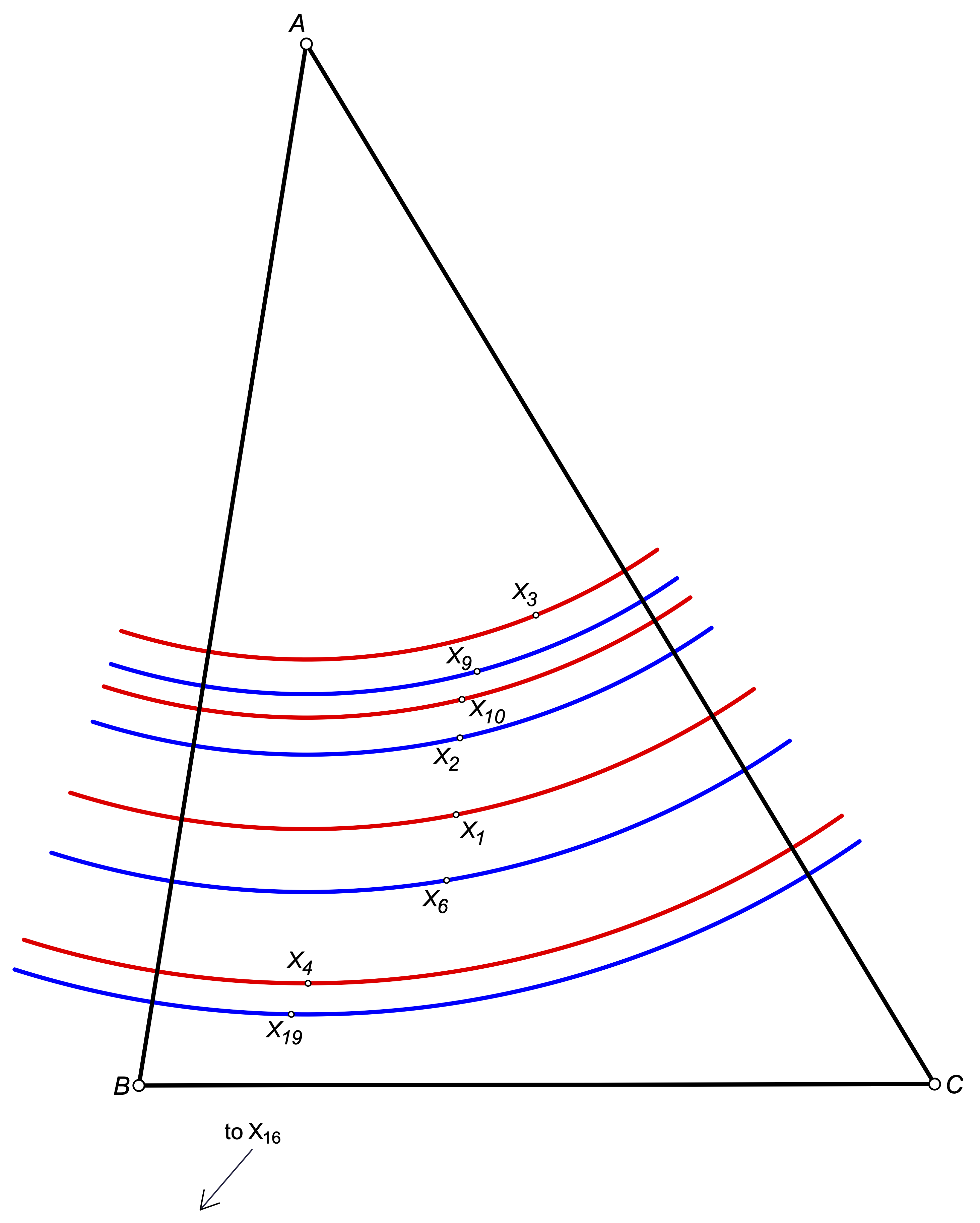}
\caption{Ordered centers in an acute triangle with shortest side $BC$}
\label{fig:vertexOrder}
\end{figure}

\begin{proof}
Let us start with the first claim, $X_{3}\prec  X_{9}$.
Let $v_n$ denote the square of the distance from $X_n$ to $A$.
It suffices to show that $v_3<v_9$
in all acute triangles $ABC$ with shortest side $BC$.

We get the barycentric coordinates for $X_{3}$ and $X_{9}$ from \cite{ETC}.
The value of $v_n$ is found using the distance formula \cite[\S2]{Grozdev}.
We shall write this as \texttt{v[n]} in Mathematica.

The condition that a triangle with smallest side $a$ is acute (\texttt{vertexConstraint}) can be written in Mathematica as

\begin{code}{0.15in}{5in}
\begin{verbatim}
    acuteConstraint = b^2<a^2+c^2 && c^2<a^2+b^2 && a^2<b^2+c^2
                                  && a>0 && b>0 && c>0;
    vertexConstraint = acuteConstraint && a<b && a<c;
\end{verbatim}
\end{code}

We can then issue the Mathematica command

\begin{code}{0.15in}{5in}
\begin{verbatim}
    Simplify[v[3]<v[9], vertexConstraint]
\end{verbatim}
\end{code}

and Mathematica responds with \texttt{True}, thus proving the inequality.

Using this \texttt{Simplify} command in Mathematica, all the other claims can be proven in the same manner.
\end{proof}

In Figure~\ref{fig:vertexOrder}, the spacing between the circular bands vary as we vary
the shape of the acute triangle. However, the order of the bands remains the same.

Not all centers from $X_1$ through $X_{20}$ appear in Theorem~\ref{thm:vertex} because some centers
vary their order in relationship with other centers.

The complete order for all centers
from $X_1$ to $X_{20}$  is shown in Figure~\ref{fig:vertexGraph}.

\begin{figure}[h!t]
\centering
\includegraphics[width=0.35\linewidth]{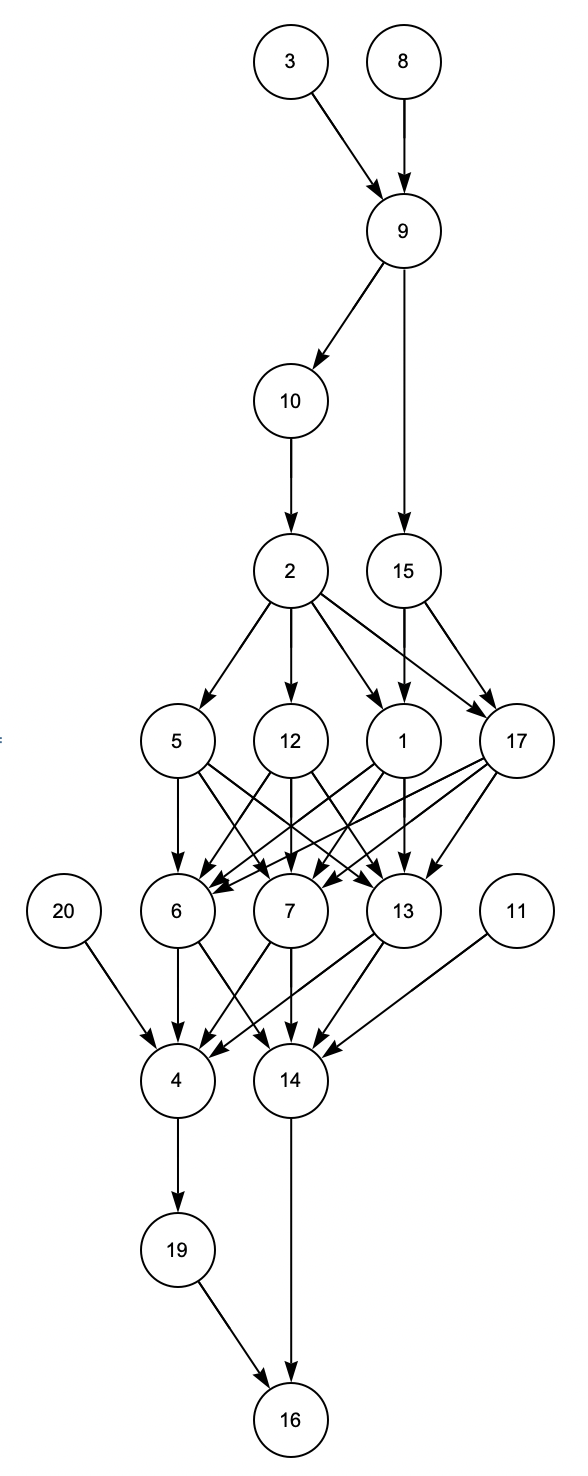}
\caption{An arrow from m to n means $X_m\prec X_n$.}
\label{fig:vertexGraph}
\end{figure}

The center $X_{18}$ is not included in this graph because its distance from vertex $A$ varies wildly.
Depending on the shape of the triangle, $X_{18}$ can get closer to vertex $A$ than either $X_3$ or $X_8$ and
can get farther away from $A$ than $X_{16}$.

Direct substitution confirms the following two results.

\begin{theorem}
Center $X_{18}$ coincides with vertex $A$ in a triangle with sides $a=8$, $b=15$, and
$c=\sqrt{(353+15\sqrt{93})/2}$.
\end{theorem}

\begin{theorem}
Center $X_{18}$ lies on the line at infinity in a triangle with sides
$a=16513$, $b=42189$, and $c=7 \sqrt{35 \left(940379+2 \sqrt{10302477117}\right)}$.
\end{theorem}

\newpage

%

\section{The Side Order}

\textbf{Definition.}
We define the \emph{side order} on triangle centers, $\prec$, by

\centerline{$P\prec Q$ if $P$ is further from $BC$ than $Q$ }
\centerline{in all acute triangles $ABC$ with smallest side $BC$.}

By ``$P$ is further from $BC$ than $Q$'' we mean that $d(P)>d(Q)$
where $d(X)$ denotes the signed distance from $X$ to $BC$.

Under this ordering, we find some order among triangle centers as shown by the following theorem
which involves some of the triangle centers from $X_1$ to $X_{40}$.

\begin{theorem}
\label{thm:acute}
Using the side order, $\prec$, we have
$$X_{26}\prec X_{20}\prec X_{22}\prec X_{40}\prec X_{8}\prec  X_9 \prec X_{10}\prec  X_2\prec X_{37}\prec X_1\prec  X_7\prec X_{29}$$
$$\prec X_{33}\prec X_4\prec  X_{27}\prec  X_{19}\prec  X_{28}\prec  X_{25}\prec  X_{34}\prec X_{24}\prec 
X_{36}\prec X_{16}.$$
\end{theorem}

\begin{proof}
Let us start with the first claim, $X_{26}\prec  X_{20}$.
We need to show that $d(X_{26})>d(X_{20})$
in all acute triangles $ABC$ with shortest side $BC$.

We get the barycentric coordinates for $X_{26}$ and $X_{20}$ from \cite{ETC}.
From Theorem~\ref{thm:signedDistance}, we find their signed distances from $BC$
which we shall call \texttt{d[26]} and \texttt{d[20]} in Mathematica.

The condition that a triangle with smallest side $a$ is acute (\texttt{sideConstraint}) can be written in Mathematica as
\begin{verbatim}
    acuteConstraint = b^2<a^2+c^2 && c^2<a^2+b^2 && a^2<b^2+c^2
                                  && a>0 && b>0 && c>0;
    sideConstraint = acuteConstraint && a<b && a<c;
\end{verbatim}
We can then issue the Mathematica command\\
\centerline{\texttt{Simplify[d[26]>d[20], sideConstraint]}}
and Mathematica responds with \texttt{True}, thus proving the inequality is always true.

Using this \texttt{Simplify} command in Mathematica, all the other claims can be proven in the same manner.
\end{proof}

What this means, graphically, is that in Figure~\ref{fig:acute},
as we vary the shape of $\triangle ABC$ (keeping it acute with smallest side $BC$),
the spacing between the horizontal lines changes, but the order of the horizontal lines remains fixed.

\begin{figure}[h!t]
\centering
\includegraphics[width=0.7\linewidth]{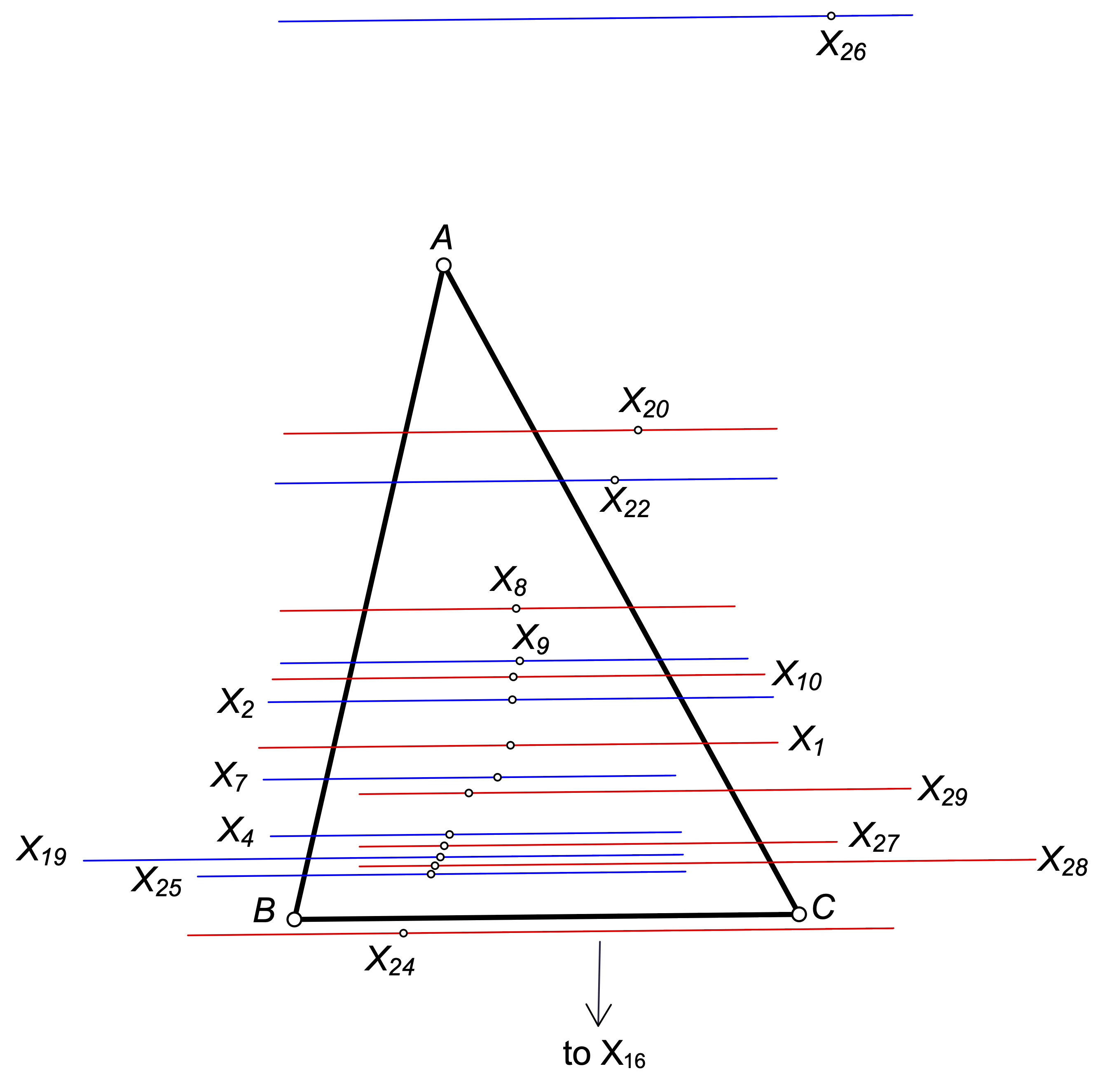}
\caption{Centers in an acute triangle with shortest side $BC$}
\label{fig:acute}
\end{figure}

Using Mathematica to compare pairs of triangle centers, we can draw a graph showing the ordering relationship
between various centers using the side order.
The graph is shown in Figure~\ref{fig:sideGraph},
where an arrow from m to n means that $X_m\prec X_n$ in the side order.

\begin{figure}[h!t]
\centering
\includegraphics[width=0.5\linewidth]{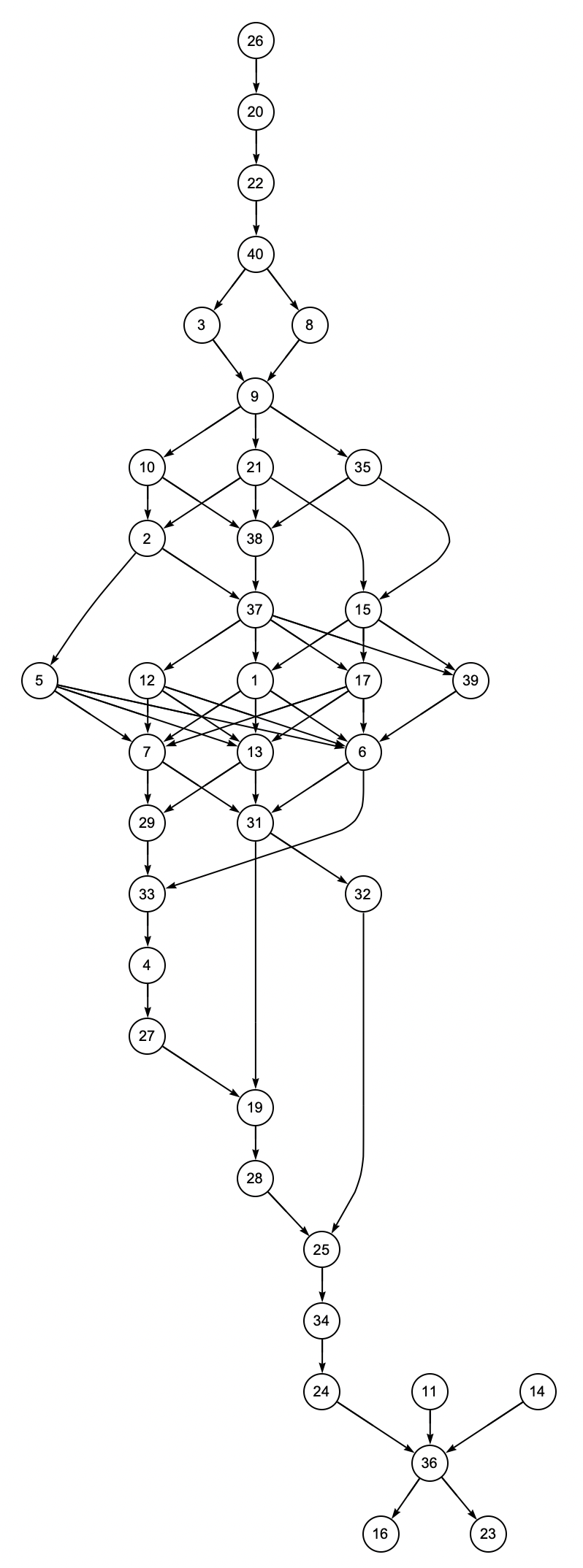}
\caption{An arrow from m to n means $X_m\prec X_n$ in the side order.}
\label{fig:sideGraph}
\end{figure}

The center $X_{18}$ is not included in this graph because it can get arbitrarily far
away from $BC$ in either direction.
The center $X_{30}$ is not included in this graph because it lies on the line at infinity.

\newpage

%

\section{The Trace Order}

A \emph{cevian} of a triangle is a line through one of the vertices.
The \emph{trace} of the cevian is the point where that line meets the side opposite to that vertex.

\begin{figure}[h!t]
\centering
\includegraphics[width=0.4\linewidth]{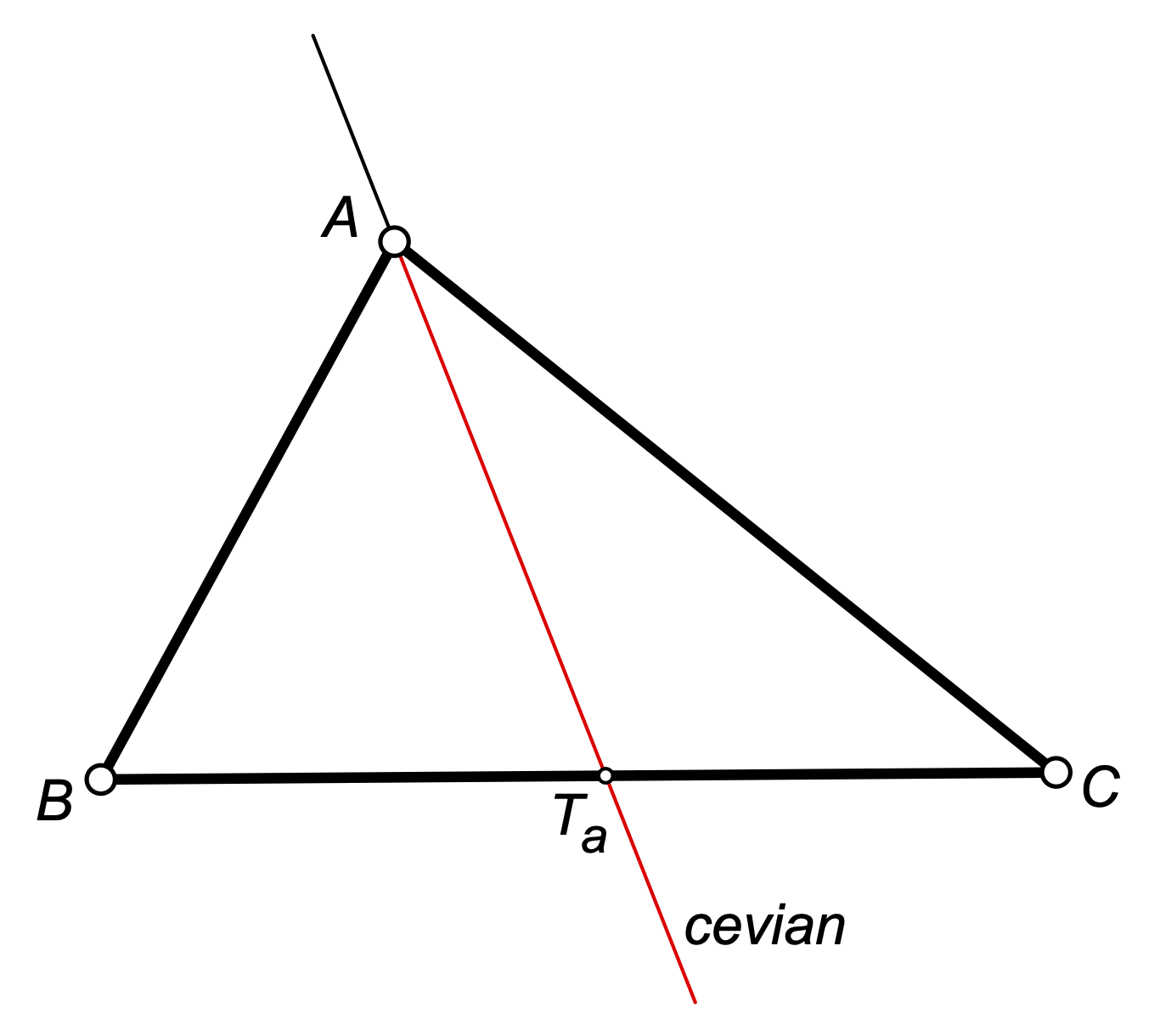}
\caption{Cevian with trace $T_a$}
\label{fig:cevian}
\end{figure}

Figure~\ref{fig:cevian} shows a cevian through vertex $A$ and its trace $T_a$.
Note that the trace can lie anywhere on line $BC$ and is not necessarily between $B$ and $C$.

If $P$ is a point in the plane of $\triangle ABC$, distinct from $A$, then the trace of
the cevian $AP$ is called the \emph{A-trace} of $P$ and will be denoted by $\hbox{Atrace}(P)$.
In Figure~\ref{fig:traceOrderDef}, $T_P=\hbox{Atrace}(P)$.

\begin{figure}[h!t]
\centering
\includegraphics[width=0.4\linewidth]{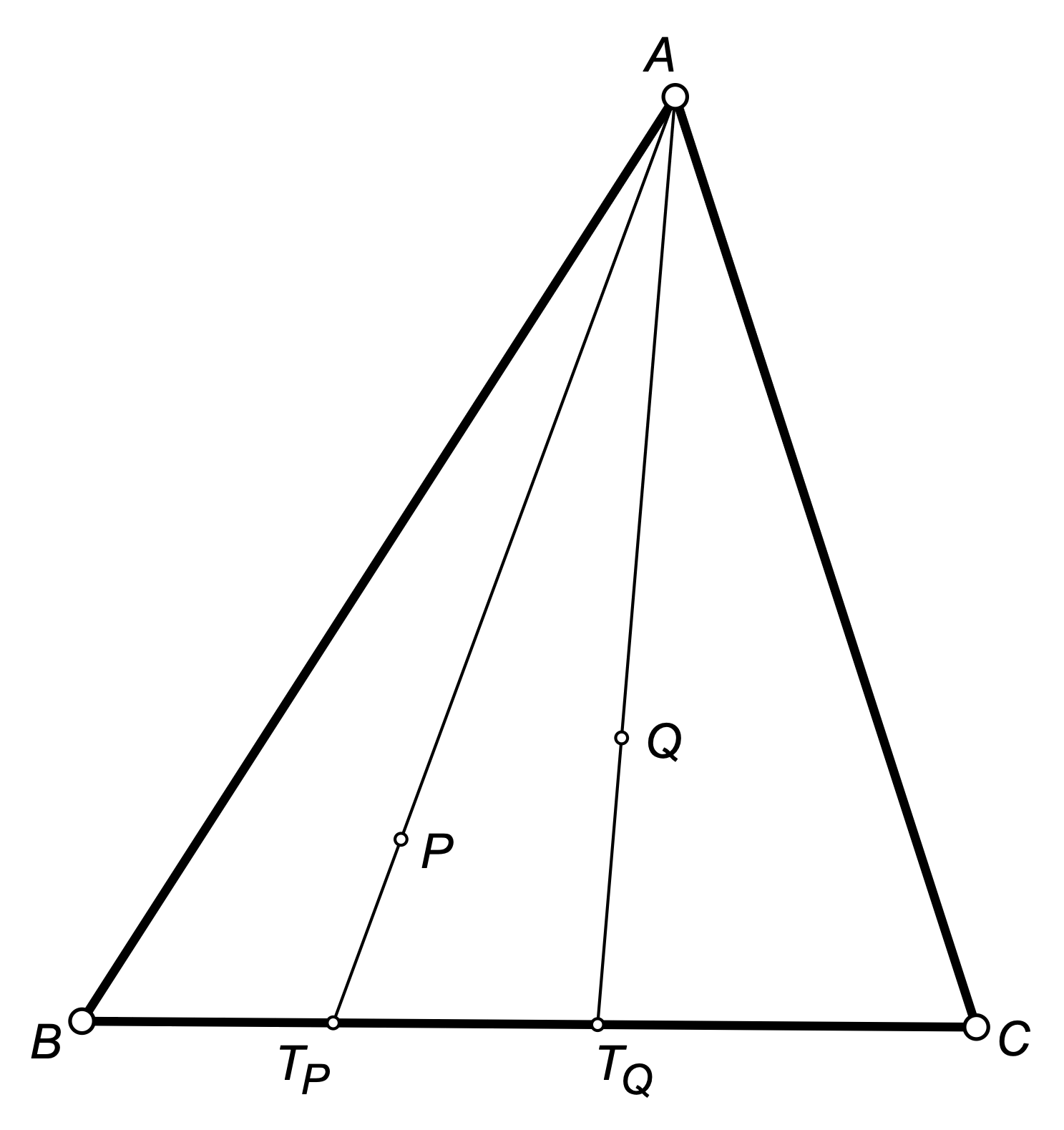}
\caption{$P\prec $Q}
\label{fig:traceOrderDef}
\end{figure}

\begin{lemma}
\label{lemma:unsignedAtrace}
If $P$ has barycentric coordinates $(p:q:r)$, then the distance from $\hbox{Atrace}(P)$ to $C$ is
$$\left|\frac{aq}{q+r}\right|.$$
\end{lemma}

\begin{proof}
Using the intersection formula (5) in \cite{Grozdev}, we find that point $T_P=\hbox{Atrace}(P)$
has coordinates $(0:q:r)$. Using the distance formula (9) in  \cite{Grozdev}, we find that
the square of the distance from $T_P$ to $C$ is $a^2q^2/(q+r)^2$.
\end{proof}

\begin{lemma}
\label{lemma:right}
Let $ABC$ be a triangle, and let $P$ be a point on the line $BC$ with barycentric coordinates
$P=(0:v:w)$.
Then $P$ lies on the extension of $BC$ beyond $C$ if and only if
$$v(v+w)<0.$$
\end{lemma}

\begin{proof}
The normalized coordinates for $P$ are
$$\left(0:\frac{v}{v+w}:\frac{w}{v+w}\right).$$
Looking at Figure~\ref{fig:7areas}, we see that a point with normalized coordinates $(0:p:q)$
lies on the extension of $BC$ beyond $C$ if and only if $q<0$.
Thus, $P$ lies on the extension of $BC$ beyond $C$ if and only if $v/(v+w)<0$.
Multiplying both sides of this inequality by the positive quantity $(v+w)^2$, we see
that the condition is $v(v+w)<0$.
\end{proof}

\textbf{Definition.} If $X$ lies on line $BC$, then the \emph{signed distance} from $X$ to $C$
is the actual distance (taken as positive) if $X$ lies on $\overrightarrow{CB}$ and the negative of the distance
if $X$ lies on the extension of $BC$ beyond $C$.

\begin{lemma}
\label{lemma:signedAtrace}
If $P$ has barycentric coordinates $(p:q:r)$, then the signed distance from $\hbox{Atrace}(P)$ to $C$ is
$$\frac{aq}{q+r}.$$
\end{lemma}

\begin{proof}
By Lemma~\ref{lemma:unsignedAtrace}, the unsigned distance is 
$$\left|\frac{aq}{q+r}\right|=a\left|\frac{q}{q+r}\right|.$$
By Lemma~\ref{lemma:right}, if $P$ lies on the extension of $BC$ beyond $C$,
then $q$ and $q+r$ must have opposite signs, so the signed distance is negative and
equal to $aq/(q+r)$. Similarly, if $P$ lies on $\overrightarrow{CB}$, then the signed distance
is positive, and by Lemma~\ref{lemma:right}, $q$ and $q+r$ have the same sign, making
the signed distance equal to $aq/(q+r)$.
\end{proof}

We say that a point $P$ on line $BC$ is \emph{to the right of $C$} if
$P$ lies on the extension of $BC$ beyond $C$.

\newpage

\textbf{Definition.}
We define the \emph{trace order} on triangle centers, $\prec$, by

\centerline{$P\prec Q$ if $\hbox{Atrace}(P)$ is further from $C$ than $\hbox{Atrace}(Q)$ }
\centerline{in all acute triangles $ABC$ with $a<b<c$.}

By ``$U$ is further from $C$ than $V$'' we mean that $d(U)>d(V)$
where $d(X)$ denotes the signed distance from $X$ to $C$.

Considering only the triangle centers from $X_1$ through $X_{30}$, we get the following result.

\begin{theorem}
\label{thm:acuteTrace}
Using the trace order $\prec$, we have
$$X_{20}\prec X_{22}\prec  X_3\prec X_{8}\prec  X_9 \prec X_{21}\prec X_{10}\prec  X_2\prec X_1\prec X_{17}\prec X_{12}$$
$$\prec  X_7\prec X_{13}\prec X_{29}\prec X_4\prec X_{27}\prec X_{19}\prec X_{28}\prec X_{25}\prec X_{24}\prec X_{23}.$$
\end{theorem}

\begin{proof}
Let $d_n$ denote the signed distance from $\hbox{Atrace}(X_n)$ to $C$.
From Lemma~\ref{lemma:signedAtrace}, the Mathematica function for $d_n$ can be written as
\begin{verbatim}
     d[{p_, q_, r_}] := a*q/(q+r);
     d[n_] := d[x[n]];
\end{verbatim}
where \texttt{x[n]} represents the barycentric coordinates for $X_n$.
Let us start with the first claim, $X_{20}\prec  X_{22}$.
The barycentric coordinates \texttt{x[20]} and \texttt{x[22]} are obtained from \cite{ETC}.
We can then execute the Mathematica commands
\begin{verbatim}
    acuteConstraint = b^2<a^2+c^2 && c^2<a^2+b^2 && a^2<b^2+c^2
                                  && a>0 && b>0 && c>0;
    traceConstraint = acuteConstraint && a<b<c;
    Simplify[d[20] > d[22], traceConstraint]
\end{verbatim}
and Mathematica responds with \texttt{True}, proving that $\hbox{Atrace}(X_{20})$
is always further from $C$ than $\hbox{Atrace}(X_{22})$.
The other claims are proven in the same manner.
\end{proof}

Figure~\ref{fig:traceOrder} shows some of the centers from $X_1$ through $X_{19}$. As we vary the shape of the acute triangle $ABC$ (keeping $a<b<c$),
the spacing of the traces will vary, but their order remains fixed.

\begin{figure}[h!t]
\centering
\includegraphics[width=0.5\linewidth]{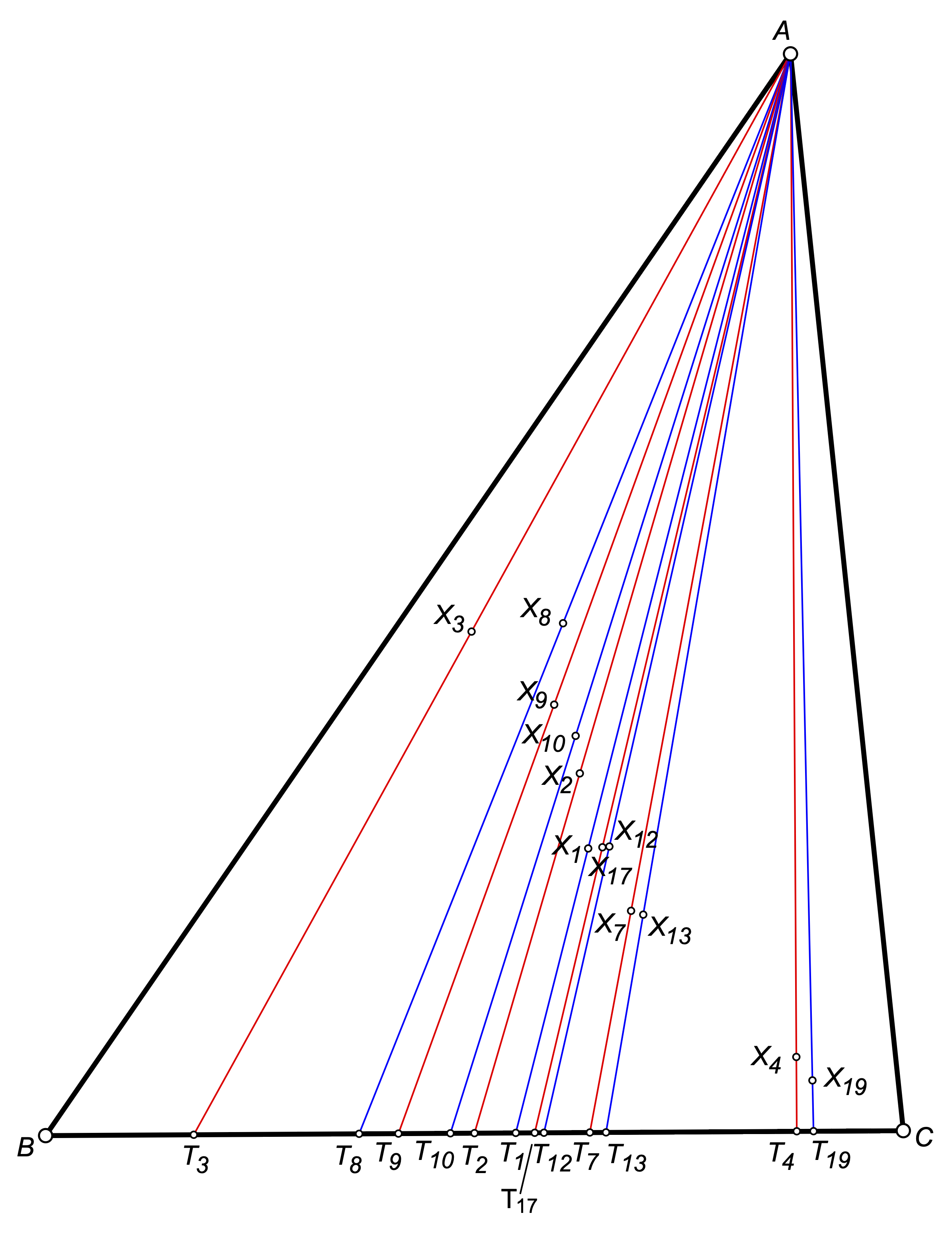}
\caption{Centers in an acute triangle and their traces}
\label{fig:traceOrder}
\end{figure}

Figure~\ref{fig:traceGraph} shows the full order.

\begin{figure}[h!t]
\centering
\includegraphics[width=0.4\linewidth]{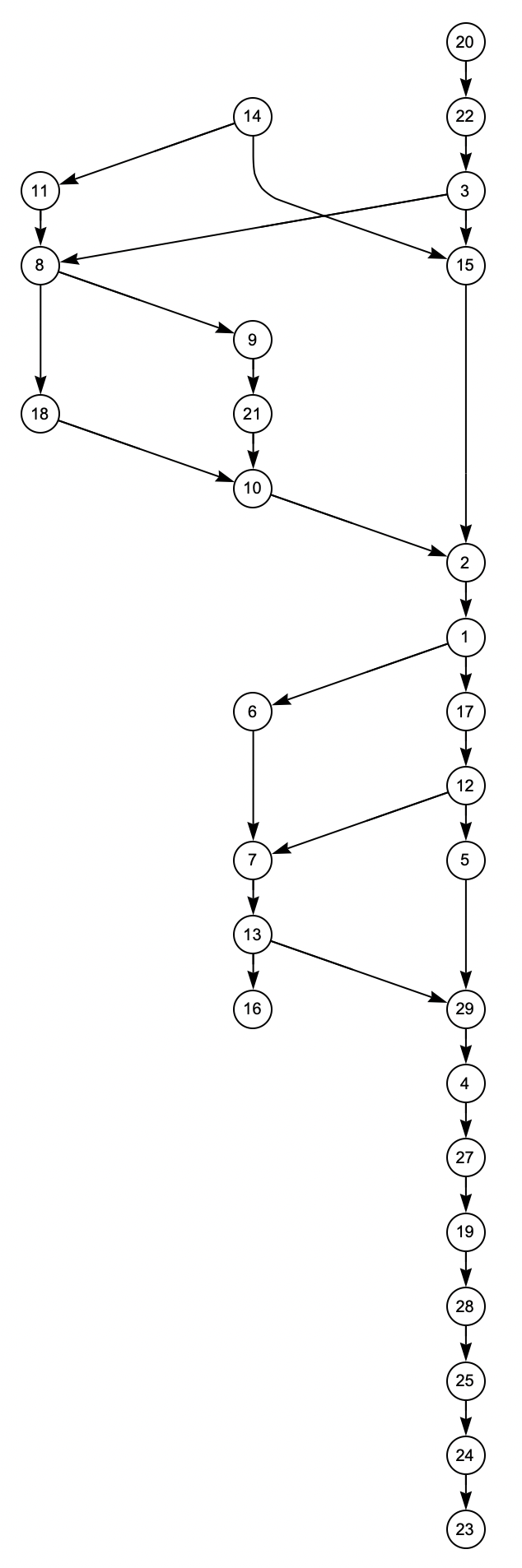}
\caption{An arrow from m to n means $X_m\prec X_n$ under the trace order.}
\label{fig:traceGraph}
\end{figure}

The center $X_{26}$ does not appear in this graph because its A-trace can be arbitrarily far from $C$ in either direction.
The center $X_{30}$ does not appear in this graph because it lies on the line at infinity.

\begin{open}
Is there a simple reason that center $X_1$ is a cutpoint
for the graph shown in Figure~\ref{fig:traceGraph}?
\end{open}

\begin{open}
As more centers are added to the graph in Figure~\ref{fig:traceGraph},
does $X_1$ remain a cutpoint?
\end{open}

Figure~\ref{fig:traceOrder} suggests that the traces of all centers lie on line segment $BC$.
This is not the case.

\goodbreak

\begin{theorem}
Using the trace order $\prec$, we have $C\prec X_{24}$.
\end{theorem}

\begin{proof}
We execute the Mathematica command
\begin{verbatim}
    Simplify[d[24] > d[ptC], traceConstraint]
\end{verbatim}
where \texttt{ptC} $=\{0,0,1\}$ and
Mathematica responds with \texttt{True}, confirming that $C\prec X_{24}$.
\end{proof}

The following theorem is proven in the same manner.

\begin{theorem}
Using the trace order $\prec$, we have $X_{650}\prec B$.
\end{theorem}

\begin{theorem}
The A-trace of the center $X_{23}$ can occur to the right of $C$.
\end{theorem}

\begin{figure}[h!t]
\centering
\includegraphics[width=0.3\linewidth]{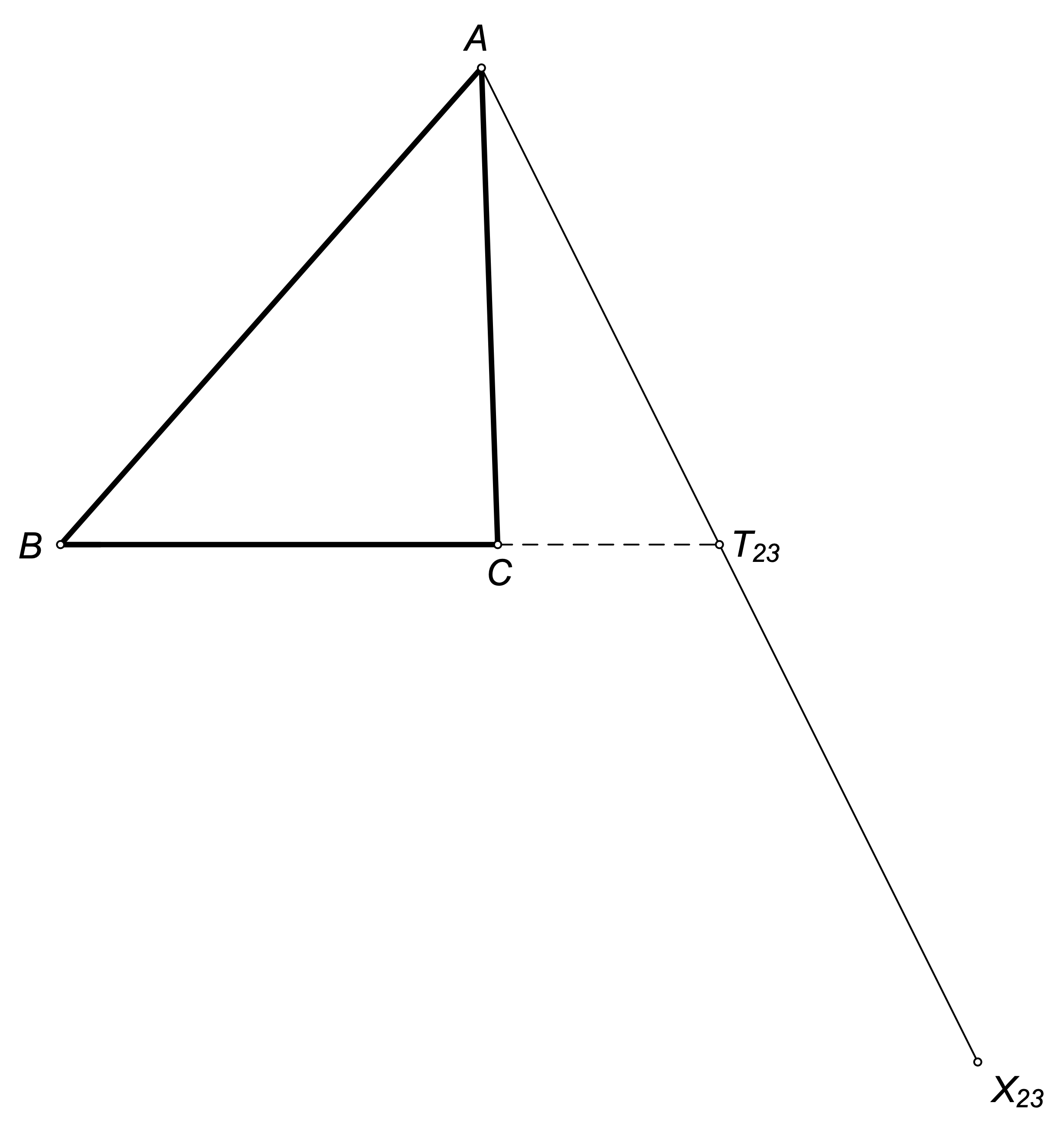}
\caption{An 11--12--16 triangle}
\label{fig:X23}
\end{figure}

\begin{proof}
See figure~\ref{fig:X23} showing an 11--12--16 triangle in which $T_{23}$ is to the right of point~$C$.
Note that angle $C$ is about $88\degrees$.
\end{proof}

Although not shown in Figure~\ref{fig:X23}, the same triangle shows that $\hbox{Atrace}(X_i)$
can lie to the right of $C$ for $i\in\{23,36,44,50,64,66,84,99,100\}$.

\begin{theorem}
Of the first 29 triangle centers, only $X_{16}$, $X_{23}$, and $X_{26}$ can have an A-trace
to the right of $C$.
\end{theorem}

\begin{proof}
From Lemma~\ref{lemma:right}, the Mathematica condition for the A-trace to be to the right of $C$ is
\begin{verbatim}
     toTheRightOfC[{p_, q_, r_}] := q(q+r) < 0
\end{verbatim}
and we can then use the Mathematica command
\begin{verbatim}
    FindInstance[toTheRightOfC[x[i]] && traceConstraint, {a,b,c}]
\end{verbatim}
(where \texttt{x[i]} denotes the coordinates for center $X_i$) to determine if $\hbox{Atrace}(X_i)$
can be to the right of $C$. As we vary $i$ from 1 to 29, we find that the only values for $i$
for which $\hbox{Atrace}(X_i)$ can be to the right of $C$ are 16, 23, and 26.
\end{proof}

\begin{theorem}
In a triangle with sides $a=1$, $b=2$, and $c=\sqrt{(1+\sqrt{61})/2}$, the A-trace of center $X_{23}$
coincides with $C$. In other words, $X_{23}$ lies on side $AC$ for that triangle.
\end{theorem}

\begin{proof}
From \cite{ETC}, we find that the barycentric coordinates for $X_{23}$ are
$$\Bigl(a^2 \left(a^4-b^4+b^2 c^2-c^4\right):b^2 \left(-a^4+a^2 c^2+b^4-c^4\right):c^2
   \left(-a^4+a^2 b^2-b^4+c^4\right)\Bigr).$$
Letting $a=1$, $b=2$, and $c=\sqrt{(1+\sqrt{61})/2}$ transforms this to
$$\left(\frac{3}{2} \left(\sqrt{61}-19\right):0:\frac{3}{2} \left(11+\sqrt{61}\right)\right)$$
which shows that $X_{23}$ lies on $AC$.
\end{proof}

\section{Areas for Future Research}

Take some line associated with a triangle, such as the Euler line, the Brocard Axis, or a symmedian.
Project centers onto this line using a point projection, a parallel projection, or an orthogonal projection.
Investigate the order of the traces on this line for special types of triangles, such as acute triangles
or triangles with a $60\degrees$ angle.

\section{Concluding Remarks}

These ordering relations suggest that triangle centers possess an unexpected global structure, analogous to the partial orders that arise in other areas of discrete geometry.

We hope this paper has made the reader see that there is more order amongst triangle centers
than they may have thought.


\end{document}